\documentclass[12pt]{article}
\usepackage{graphicx} % Required for inserting images
\usepackage{amssymb}
\usepackage{graphicx}
\usepackage{graphics}
\usepackage{epsfig}
\usepackage{float}
\usepackage{subcaption}
\usepackage{amsmath}
\usepackage{amsthm}
\usepackage[T1]{fontenc}
\usepackage{accents}
\usepackage[english]{babel}

\newtheorem{theorem}{Theorem}[section]
\newtheorem{proposition}{Proposition}[section]
\newtheorem{lemma}{Lemma}[section]
\theoremstyle{remark}
\newtheorem{remark}{Remark}[section]
\theoremstyle{definition}
\newtheorem{exmp}{Example}[section]
\newtheorem{assumption}{Assumption}
\usepackage{amsmath}
\usepackage{xcolor}
\usepackage{framed}
\usepackage{soul}
\usepackage{color}
\usepackage{multirow}
\usepackage{array}
\usepackage{enumerate}
\newcolumntype{P}[1]{>{\centering\arraybackslash}p{#1}}

\usepackage{authblk}
\usepackage{geometry}
\numberwithin{equation}{section}

\title{Convergence analysis for a tree-based nonlinear reduced basis method}
\author[1]{Mohamed Barakat\thanks{M.~Barakat and D.~Guignard are partially supported by the NSERC Grant RGPIN-2021-04311.}}
\author[1]{Diane Guignard } 
\affil[1]{\small Department of Mathematics and Statistics, University of Ottawa, Ottawa, ON K1N 6N5, Canada}
\date{}

\begin{document}

\maketitle
\begin{abstract}
    We develop and analyze a nonlinear reduced basis (RB) method for parametrized elliptic partial differential equations based on a binary-tree partition of the parameter domain into tensor-product structured subdomains. Each subdomain is associated with a local RB space of prescribed dimension, constructed via a greedy algorithm. A splitting strategy along the longest edge of the parameter subdomains ensures geometric control of the subdomains and enables a rigorous convergence analysis. Under the assumption that the parameter-to-solution map admits a holomorphic extension and that the resulting domain partition is quasi-uniform, we establish explicit bounds on the number of subdomains required to achieve a given tolerance for arbitrary parameter domain dimension and RB spaces size. Numerical experiments for diffusion and convection–diffusion problems confirm the theoretical predictions, demonstrating that the proposed approach, which has low storage requirements, achieves the expected convergence rates and in several cases outperforms an existing nonlinear RB method.
\end{abstract}

\section{Introduction}
Reduced order modeling has become a central tool in many numerical simulation workflows. The reduced basis (RB) method, introduced in \cite{almroth1978automatic, noor1980reduced} and further analyzed in \cite{porsching1985estimation, rheinboldt1992theory}, is widely used to reduce the computational cost of solving parametrized partial differential equations (PDEs). It is particularly effective in real-time applications \cite{nguyen2010reduced}, where the solution of the problem needs to be known very quickly under limited resources for a previously unknown parameter, as well as in many-query contexts \cite{boyaval2008reduced,boyaval2009reduced}, where the problem has to be solved repeatedly for a range of parameter values. 

The method constructs a low dimensional linear approximation space $V_N$, spanned by a set of high-fidelity (e.g., accurate finite element (FE)) solutions, referred to as snapshots, which are computed during an offline stage at optimally selected parameter values. These snapshots capture the dominant features of the solution manifold and serve as the foundation for efficient online evaluations in the reduced order simulations. During the online stage, for any given input parameter, the RB approximation is formed as the Galerkin projection onto the space $V_N$ \cite{rozza2008reduced}. 

In many of the cases where the solution changes smoothly with respect to the parameters, linear RB methods can achieve high accuracy with a small number of snapshots, allowing for exponential convergence rates. However, in cases where the solution has very different structures in different regions of the parameter domain, such as in transport equations \cite{algoritmy} or problems with highly variable parameters \cite{eftang2010hp}, the linear RB space may require a large number of snapshots. This results in a high computational cost during the online evaluation stage. This issue motivates the need for nonlinear model reduction techniques.

Nonlinear reduced order modeling has been discussed in the context of RB methods, for instance in \cite{eftang2010hp,maday2013locally,zou2019adaptive,bonito2021nonlinear,cohen2022nonlinear}. The nonlinear RB method presented in this work is a "library approximation" reduced order modeling where the linear space $V_N$ is replaced by a collection of linear spaces called a library. In this work, we build on the strategy introduced in \cite{eftang2010hp} where the main idea is to recursively split the parameter domain $\mathcal{D} \subset \mathbb{R}^d$ into smaller subdomains and associate them to a binary tree in the offline stage. A linear RB model with a chosen space size, $N$, is associated to each subdomain. If the error estimator over a subdomain satisfies a prescribed tolerance, this branch of the tree is terminated and no further splitting is required. Otherwise, the subdomain is further split into two subdomains according to a specific criterion. Once all the subdomains satisfy the tolerance, we have access to a library of RB spaces based on the partition of the original parameter domain. During the online stage, given a new parameter $\mu \in \mathcal{D}$, the subdomain containing this parameter is identified and the corresponding RB model is used to approximate the solution at the parameter.

The available convergence results provided in \cite{eftang2010hp}, which concern the number of parameter subdomains required to achieve a desired level of accuracy, are limited to the case of a one-dimensional parameter domain ($d = 1$) and the zeroth-order ($N = 1$) RB approximation. Our main goal here is to provide theoretical convergence results for the general case $d \geq 1$ and $N \geq 1$. To this end, we propose a new parameter domain partitioning technique and impose additional assumptions on the parameter-to-solution map $\mu \mapsto u(\mu)$.

The remainder of this paper is organized as follows. In Section 2, we introduce the parametric PDE problem and its discretization.  Section 3 reviews the fundamentals of the linear reduced basis method, including the greedy algorithm and a posteriori error estimation. In Section 4, we present the proposed nonlinear reduced basis approach, including the binary-tree partitioning algorithm and convergence analysis. Section 5 reports numerical experiments for diffusion and convection–diffusion problems, validating the theoretical results and evaluating the performance of the method. Finally, conclusions are drawn in Section 6.

\section{Problem statement}
We consider linear, elliptic, coercive PDE. Let $\Omega \subset \mathbb{R}^n$ with Lipschitz boundary $\partial \Omega$ denotes the physical domain. We consider a space $V=V(\Omega)$ such that $H_0^1(\Omega) \subset  V \subset H^1(\Omega)$. We equip $V$ with the inner product $(\cdot,\cdot)_v$, inducing the norm $\|\cdot\|_V$. We introduce a compact parameter domain $\mathcal{D} \subset \mathbb{R}^d$ such that a point in $\mathcal{D}$ is denoted by $\mu=(\mu_1,\dots, \mu_d$). Given a parametrized linear form $f: V \times \mathcal{D} \rightarrow \mathbb{R}$ where the linearity is with respect to the first variable, and a bilinear form $a: V \times V \times \mathcal{D} \rightarrow \mathbb{R}$ where the bilinearity is with respect to the first two variables, the abstract parametric PDE reads: Given $\mu \in \mathcal{D}$ find $u(\mu) \in V$ such that
\begin{equation}
    a(u(\mu), v ; \mu)=f(v ; \mu), \quad \forall v \in V.
    \label{para pde}
\end{equation}
To ensure an efficiency through an offline–online process, the forms $a(.,.;\mu)$ and $f(.;\mu)$ are assumed to admit the affine decomposition
\begin{equation}
    a(w, v ; \mu) =\sum_{q=1}^{Q_{\mathrm{a}}} \theta_{\mathrm{a}}^q(\mu) a_q(w, v), \quad
    f(v ; \mu) =\sum_{q=1}^{Q_{\mathrm{f}}} \theta_{\mathrm{f}}^q(\mu) f_q(v), \label{aff 2} 
\end{equation}
where the forms $a_q(\cdot, \cdot)$ and $f_q(\cdot)$ are $\mu-$independent, and the $\theta_{\mathrm{a}}^q$ and $\theta_{\mathrm{f}}^q$ are $\mu-$dependent continuous functions. When an affine decomposition is not readily available, the Empirical Interpolation Method (EIM) can be employed to construct an accurate affine approximation of the bilinear and linear forms $a$ and $f$, respectively. For a detailed discussion of the EIM, the reader is referred to \cite{barrault2004empirical, grepl2007efficient,eftang2010posteriori}.
According to the Lax-Milgram theorem \cite{quarteroni2008numerical}, Problem \eqref{para pde} is well posed if for all parameters $\mu \in \mathcal{D}$, the bilinear form $a(\cdot, \cdot ; \mu)$ is coercive and continuous, and the linear form $f(\cdot ; \mu)$ is continuous with respect to the norm $\|\cdot\|_V$. The coercivity and continuity constants of $a(\cdot, \cdot ; \mu)$ with respect to the norm $\|\cdot\|_V$ are defined as 
\begin{equation}
    \alpha(\mu):=\inf _{v \in V} \frac{a(v, v ; \mu)}{\|v\|_{V}^2}, \quad  \quad \gamma(\mu):=\sup _{w \in V} \sup _{v \in V} \frac{a(w, v ; \mu)}{\|w\|_V \|v\|_{V}}.
    \label{corec}
\end{equation}

We introduce an approximation space $V_{\mathcal{N}} \subset V$ of finite dimension $\mathcal{N}$. Throughout this work, the approximation space $V_{\mathcal{N}}$ is taken to be a standard finite element space. The discretization of \eqref{para pde} is given by: Given $\mu \in \mathcal{D}$, find $u_{\mathcal{N}}(\mu) \in V_{\mathcal{N}}$ such that
\begin{equation}
    a(u_{\mathcal{N}}(\mu), v ; \mu)=f(v ; \mu), \quad \forall v \in V_{\mathcal{N}}.
    \label{dis prob}
\end{equation}
Problem \eqref{dis prob} is referred to as the truth problem or the high-fidelity model. The space $V_{\mathcal{N}}$ is assumed to be fine enough so that the error between the exact solution and the truth solution is negligible. Due to the high dimension of the approximation space $V_{\mathcal{N}}$ meant to achieve high accuracy, the computation of the truth solution is potentially very expensive.

%%%%%%%%%%%%%%%%%%%%%%%%%%%%%%%%%%%%%%%%%%%%%%%%%%%%%%%%%%%%%%%%%%%%%%%%%%%%%%%%%
\section{Linear reduced basis method}
Here, we review the basic concepts of (linear) reduced basis methods. For more comprehensive details, we refer the reader to \cite{hesthaven2016certified,quarteroni2015reduced}. We begin by defining the solution manifold of the parametric PDE \eqref{para pde} by
\begin{equation}
    \mathcal{M}=\mathcal{M}(\mathcal{D}):=\left\{u(\mu) \mid \mu \in \mathcal{D}\right\} \subset V,
\end{equation}
where each $u(\mu) \in V$ corresponds to the solution of the exact problem. The discrete counterpart of the solution manifold is defined analogously by
\begin{equation}
    \mathcal{M}_{\mathcal{N}}=\mathcal{M}_{\mathcal{N}}(\mathcal{D}):=\left\{u_{\mathcal{N}}(\mu) \mid \mu \in \mathcal{D}\right\} \subset V_{\mathcal{N}}.
\end{equation}
The success of any reduced order model is based on the assumption that the dimension of the solution manifold is low meaning that it can be well approximated by the linear span of a small number $(N << {\mathcal{N}})$ of judiciously chosen basis functions. These basis functions are referred to as the reduced basis, and the subspace spanned by them is denoted by $V_N$. That is, assuming that the reduced basis, denoted by $\{\xi_i \}_{i=1}^N \subset V_{\mathcal{N}}$, is available, the reduced basis space is given by
\begin{equation}
    V_N=\text{span}\{\xi_1,\dots,\xi_N\} \subset V_{\mathcal{N}},
\end{equation}
and the reduced basis problem reads: Given $\mu \in \mathcal{D}$, find $u_N(\mu) \in V_N$ such that
\begin{equation}
    a\left(u_N(\mu), v ; \mu\right)=f\left(v ; \mu\right), \quad \forall v \in V_N.
    \label{rb prob}
\end{equation}
Given the following error bound
\begin{equation}
    \left\|u(\mu)-u_N(\mu)\right\|_{V} \leq\left\|u(\mu)-u_{\mathcal{N}}(\mu)\right\|_{V}+\left\|u_{\mathcal{N}}(\mu)-u_N(\mu)\right\|_{V},
\end{equation}
and assuming that the first term on the right hand side is negligible, it follows that the accuracy of the reduced order model is primarily determined by how well the reduced basis solution approximates the high-fidelity solution. It should be noted that the possibility of obtaining this accurate approximation using a low dimensional reduced basis space is problem-dependent. To get a handle on this, we recall the definition of the Kolmogorov $N$-width defined as
\begin{equation}
    d_N(\mathcal{M}_{\mathcal{N}}):= \inf _{\operatorname{dim}\left(V_N\right) \leq N} \sup _{u \in \mathcal{M}_{\mathcal{N}}} \inf _{w \in V_N}\|u-w\|_V.
\end{equation}
The Kolmogorov $N$-width quantifies how accurately $\mathcal{M}_{\mathcal{N}}$ can be approximated by a $N$-dimensional linear space $V_{N}$. If the $N$-width decays rapidly as $N$ grows, it indicates that the solution manifold can be effectively approximated by a reduced basis of small dimension, providing an efficient approximation across the entire parameter domain. Finally, we emphasize that the generation of the reduced basis space is performed in an offline stage that is usually computationally expensive and the evaluation of the solution at any given parameter is done in an online stage by solving Problem \eqref{rb prob} at a lower computational cost than solving the truth problem \eqref{dis prob}.

%%%%%%%%%%%%%%%%%%%%%%%%%%%%%%%%%%%%%%%%%%%%%%%%%%%%%%%%%%%%%%%%%%%%%%%%%%%%%%%%%%%%%%%%%%%%%%
\subsection{Greedy algorithm}
There are several approaches for generating reduced basis spaces. In this work, we adopt the greedy construction as it comes with a rigorous convergence theory \cite{binev2011convergence,devore2013greedy} and its iterative nature ensures its computational efficiency. We begin by introducing a discrete finite-dimensional training set $E\subset\ \mathcal{D}$. The training set can consist of a regular lattice or a randomly generated point-set intersecting with $\mathcal{D}$. We define the solution manifold of the training set as 
\begin{equation}
    \mathcal{M}_{\mathcal{N}}(E):=\left\{u_{\mathcal{N}}(\mu) \mid \mu \in E\right\} \subset V_{\mathcal{N}},
\end{equation}
and if $E$ is fine enough, then $\mathcal{M}_{\mathcal{N}}(E)$ is a good representation of $\mathcal{M}_{\mathcal{N}}(\mathcal{D})$. 

The reduced basis space is built upon truth snapshots $u_{\mathcal{N}}(\mu^{(n)}), 1 \leq n \leq N$, for some $\mu^{(1)},\dots, \mu^{(N)} \in E$ that are successively selected by the greedy algorithm. The algorithm is an iterative process where at each step a new basis function is introduced, improving the precision of the basis. In the $n$th step, an $n-$dimensional reduced basis space is given and the next snapshot is the one that maximizes the error of the current RB space over the training set $E$ (since it is impossible to evaluate the error across the entire parameter space $\mathcal{D}$, the need for the training set $E$ arises). That is, we begin by randomly selecting the first parameter $\mu^{(1)}$ from the training set. Then, at the general step, we select 
\begin{equation}
    \mu^{(n+1)}=\underset{\mu \in E}{\arg \max }  \|u_{\mathcal{N}}(\mu)-u_n(\mu)\|_V,
    \label{gred sel}
\end{equation}
where $u_n(\mu)$ is the solution of \eqref{rb prob} with  
\begin{equation}
    V_{n}= \text{span} \{u_{\mathcal{N}}(\mu^{(1)}),\dots, u_{\mathcal{N}}(\mu^{(n)})\}
\end{equation}
in place of $V_N$. Once we reach a prescribed tolerance, we terminate the algorithm and the reduced basis space is formed based on the $N$ selected parameters as
\begin{equation}
    V_N= \text{span} \{u_{\mathcal{N}}(\mu^{(1)}),\dots, u_{\mathcal{N}}(\mu^{(N)})\}.
\end{equation}

Due to Equation \eqref{gred sel}, the greedy construction of the reduced basis requires obtaining the truth solution $u_{\mathcal{N}}(\mu)$ at all the parameter points in the training set $E$, resulting in a very computationally expensive offline stage. Therefore, a crucial element in the success of the greedy construction is the availability of an error estimator $\eta_n(\mu)$ that provides an estimate of the error due to replacing $V_{\mathcal{N}}$ by $V_{n}$. For any parameter $\mu \in \mathcal{D}$, an error estimator should satisfy
\begin{equation}
    \|u_{\mathcal{N}}(\mu)-u_n(\mu)\|_V\leq \eta_n(\mu),
\end{equation}
and the evaluation of $\eta_n$ should be less expensive than solving the truth problem. Hence, Equation \eqref{gred sel} can be replaced by 
\begin{equation}
    \mu^{(n+1)}=\underset{\mu \in E}{\arg \max } \ \eta_n(\mu).
\end{equation}
In this way, we require one truth solution $u_{\mathcal{N}}(\mu^{(n)})$ to be computed per iteration and a total of $N$ truth solutions to generate the $N-$dimensional reduced basis space. Since the cost of evaluating the error estimator is small, the training set can be dense to better represent the parameter domain $\mathcal{D}$.

From a practical perspective, it is crucial to note that the various snapshots $\{ u_{\mathcal{N}}(\mu^{(1)}), \dots, u_{\mathcal{N}}(\mu^{(N)})\}$ may be (nearly) linearly dependent, which can lead to computational instability. To address this, it is recommended to orthonormalize the snapshots to derive the basis functions $\{ \xi_1, \dots, \xi_N \}$. For instance, one can use the Gram-Schmidt orthonormalization algorithm based on the vector of degrees of freedom of the functions $( u_{\mathcal{N}}(\mu^{(n)}) )$ and the discrete scalar product of $V$-inner product.

In a more general setting, we consider a compact family $F:=\{f(\mu):\mu \in \mathcal{D}\}$ of parametrized functions in a Hilbert space $V$, for which we want to apply the greedy algorithm to find a subspace $F_N=\text{span} \{f_1, \dots , f_N\}$ that well approximate the set $F$. Let $P_N$ be the $V$-orthogonal projector onto $F_N$. Then, the greedy approximation error is defined as
\begin{equation}
    \sigma_N(F):= \sup_{f \in F} \|f - P_N f\|_V
    \label{sigma_n}
\end{equation}
which quantifies the worst-case error in approximating elements of $F$ using the subspace $F_N$. An important result that we use in our theoretical analysis is Theorem 4.4 from \cite{binev2011convergence} (adapted to the weak greedy algorithm), which provides a direct comparison between  the greedy error $\sigma_N(F)$ and the Kolmogorov $N$-width $d_N(F)$.
\begin{theorem}
Let $F$ be an arbitrary compact set in a Hilbert space $V$. For each $N = 1, 2,...,$ we have
\begin{equation}
    \sigma_N(F) \leq \frac{\kappa ^{N+1}}{\sqrt{3}} d_N(F),
    \label{th comp}
\end{equation}
where $\kappa > 2$ is a constant depending on the effectivity of the error estimator.
\end{theorem}

%%%%%%%%%%%%%%%%%%%%%%%%%%%%%%%%%%%%%%%%%%%%%%%%%%%%%%%%%%%%%%%%%%%%%%%%%%%%%%%%%%%%%%%%%%%%%%%%%%%%
\subsection{Error estimator}
The development of a residual-based a posteriori error estimator in the $V$-norm for the reduced model is presented in \cite{rozza2008reduced}. The error estimator is based on the discrete coercivity and continuity constants defined by
\begin{equation}
    \alpha_{\mathcal{N}}(\mu):=\inf _{v \in V_{\mathcal{N}}} \frac{a(v, v ; \mu)}{\|v\|_{V}^2}, \quad  \quad \gamma_{\mathcal{N}}(\mu):=\sup _{w \in V_{\mathcal{N}}} \sup _{v \in V_{\mathcal{N}}} \frac{a(w, v ; \mu)}{\|w\|_V \|v\|_{V}}.
\end{equation}
Owing to the conformity of the approximation space $V_{\mathcal{N}}$, i.e., $V_{\mathcal{N}}\subset V$, the discrete coercivity and continuity constants satisfy
\begin{equation}
    \alpha(\mu) \leq \alpha_{\mathcal{N}}(\mu), \quad \gamma_{\mathcal{N}}(\mu)\leq \gamma(\mu), 
\end{equation}
where $\alpha(\mu)$ and $\gamma(\mu)$ are defined in \eqref{corec}. For any $\mu \in \mathcal{D}$, we assume that we have access to easily computable lower and upper bounds
\begin{align}
    0<\alpha_{\mathrm{LB}}(\mu)  \leq \alpha_{\mathcal{N}}(\mu), \\
     \gamma_{\mathcal{N}}(\mu) \leq  \gamma_{\mathrm{UB}}(\mu)  < \infty.
\end{align}
Given $u_{\mathcal{N}}(\mu)$ and $u_N(\mu)$ solutions of problems \eqref{dis prob} and \eqref{rb prob}, respectively, the error $e(\mu):= u_{\mathcal{N}}(\mu)-u_N(\mu)$ satisfies
\begin{equation}
    a(e(\mu),v;\mu)=r_N(v;\mu), \quad \forall v \in V_{\mathcal{N}},
\end{equation}
where $r_N(\cdot;\mu) \in V_{\mathcal{N}}'$ is the residual 
\begin{equation}
    r_N\left(v ; \mu\right):=f\left(v ; \mu\right)-a\left(u_N(\mu), v ; \mu\right), \quad \forall v \in V_{\mathcal{N}}.
\end{equation}
The Riesz representation of the residual, $R_N(\mu) \in V$, satisfies 
\begin{equation}
    \left(R_N(\mu),v\right)_V=r_N(v;\mu), \quad \forall v \in V_{\mathcal{N}}.
\end{equation}
Then, the error equation can be written as
\begin{equation}
    a(e(\mu),v;\mu)=\left(R_N(\mu),v\right)_V, \quad \forall v \in V_{\mathcal{N}}.
    \label{error}
\end{equation}
Now, we can state the following result:
\begin{lemma}
    For any $\mu \in \mathcal{D}$, the RB error estimator 
    \begin{equation}
        \eta_N(\mu)=\frac{\|R_N(\mu)\|_V}{\alpha_{\mathrm{LB}}(\mu)}
        \label{bd def}
    \end{equation}
    satisfies 
    \begin{align}
        \label{est}
        \left\|u_{\mathcal{N}}(\mu)-u_N(\mu)\right\|_V & \leq \eta_N(\mu)\leq \frac{  \gamma_{\mathrm{UB}}(\mu)}{\alpha_{\mathrm{LB}}(\mu)} \left\|u_{\mathcal{N}}(\mu)-u_N(\mu)\right\|_V
    \end{align}
    \label{lemma}
\end{lemma}
\noindent The proof of this result can be found in Lemma $3.1$ in \cite{eftang2010hp}, while the detailed computation of the estimator $\eta_N$ is presented in Section $4.2.5$ in \cite{hesthaven2016certified}.

%%%%%%%%%%%%%%%%%%%%%%%%%%%%%%%%%%%%%%%%%%%%%%%%%%%%%%%%%%%%%%%%%%%%%%%%%%%%%%%%%%%%%%%%%%%%%%%%%%%%%
\section{Nonlinear reduced basis method}

In this section, we formulate a nonlinear reduced basis method. We begin by introducing a tree structure for the subdomains. Then, we present an algorithm for partitioning the parameter domain and assigning a linear RB space to each of them. Finally, we develop our main convergence result for the number of subdomains required to achieve a certain accuracy.

\subsection{Tree-based structure for the subdomains}
Consider a binary tree of depth \( L \), which can contain up to  $K$ leaf nodes. To index the nodes at each level \( l \in \{1, \dots, L\} \), we define the set of Boolean vectors
\begin{equation}
    \mathcal{B}_l := \{1\} \times \{0,1\}^{l-1},
\end{equation}
so that any vector \( B_l \in \mathcal{B}_l \) is of the form
\begin{equation}
    B_l = (1, i_2, \dots, i_l), \quad i_j \in \{0,1\} \text{ for } j = 2, \dots, l.
\end{equation}
Each node at level \( l \) of the tree is uniquely associated with a vector \( B_l \in \mathcal{B}_l \). Tree traversal is encoded via binary extension of these vectors: concatenating a $0$ to \( B_l \) corresponds to descending to the left child, while appending a $1$ corresponds to the right child. Thus, the node represented by \( B_l \) is the parent of two children given by the vectors
\begin{equation}
    B_{l+1}^{(0)} := (B_l, 0) \quad \text{and} \quad B_{l+1}^{(1)} := (B_l, 1).
\end{equation}

We also consider a parameter domain $\mathcal{D}$ and, from this point onward, assume that $\mathcal{D}$ possesses a tensor product structure. The proposed approximation algorithm yields $K$ subdomains and associates a linear RB space to each subdomain. The subdomains preserve the same tensor product structure as the original parameter domain. Each subdomain is defined through the bounding points of the tensor product structure that is each subdomain is of form $\prod_{j=1}^d[a_j,b_j]$ with $a_j<b_j\in \mathcal{D}$, $1\leq j \leq d$. Given any parameter $\mu \in \mathcal{D}$, we determine the subdomain $\mathcal{D}_{B_l}$ that contains $\mu$ by comparing the coordinates of the parameter and the bounding points of the subdomain.

We consider a recursive domain decomposition process structured as a binary tree. The resulting subdomains at each level of the tree are denoted by
\begin{equation}
    \mathcal{D}_{B_l} \subset \mathcal{D}, \quad B_l \in \mathcal{B}_l, \quad 1 \leq l \leq L.
\end{equation}
Each subdomain $\mathcal{D}_{B_l}$  is associated with $N$ parameters $\{\mu^{(1)}_{B_l},...\mu^{(N)}_{B_l}\} \subset \mathcal{D}_{B_l}$ that are selected by the greedy algorithm. We define the linear RB space associated to the subdomain $\mathcal{D}_{B_l}$ as 
\begin{equation}
    V_{B_l,N}=\text{span} \{u_{\mathcal{N}}(\mu^{(1)}_{B_l}),...u_{\mathcal{N}}(\mu^{(N)}_{B_l})\}, \quad B_l \in  \mathcal{B}_l \quad 1 \leq l \leq L .
\end{equation}

\subsection{The approximation algorithm}
We now introduce the algorithm for partitioning the parameter domain $\mathcal{D}$. We start by introducing a training set over the original domain $\mathcal{D}_{1}=\mathcal{D}$, and we randomly choose the initial parameter from this training set. We specify the maximum RB space dimension $N$, and the error tolerance $\epsilon$, and we set $l=1$. The splitting is performed as follows.

\begin{enumerate}
    \item For the current $l$, we consider all leaf nodes $B_l \in \mathcal{B}_l$.
    
    \item For each subdomain $\mathcal{D}_{B_l}$:
    \begin{enumerate}[(i)]
        \item Introduce a finite training sample set $E_{B_l}$ and construct an RB approximation with $N$ parameter values. The parameters are chosen by the greedy algorithm with the initial parameter selected randomly from $E_{B_l}$.
        
        \item  Evaluate the maximum local error estimator $\eta_{B_l}$ over the training set of the current subdomain. The definition of the local error estimator $\eta_{B_l}$ is provided below.
        
        \item If $\eta_{B_l} \leq \epsilon$, then the subdomain needs no more further refinement, and the branch of the associated binary tree is terminated.
        
        \item If $\eta_{B_l} > \epsilon$:
            \begin{itemize}
                \item Split the subdomain $\mathcal{D}_{B_l}$ into two equi-sized subdomains $\mathcal{D}_{(B_l,0)}$, and $\mathcal{D}_{(B_l,1)}$ that preserve the tensor product structure. When $d>1$, the algorithm chooses the longest side of the subdomain and performs the split in this direction. We elaborate on this in Remark \ref{rem geo}.
                \item Split the current branch into two new branches $B_{l+1}^{(0)}=(B_{l},0)$ and $B_{l+1}^{(1)}=(B_{l},1)$.
            \end{itemize}
        \end{enumerate}
    \item Set $l=l+1$, and proceed to Step 1.
\end{enumerate}

The outcome of the algorithm is a library of $K$ RB spaces over $K$ subdomains that are associated with the $K$ leaf nodes of the binary tree. Each subdomain is defined by its tensor product structure and is associated to a linear RB space of size $N$. The intermediate subdomains, along with the linear RB spaces associated with non-leaf nodes, are discarded and do not contribute during the online stage. Furthermore, the depth of the tree, denoted by $L$, is simply the number of nodes in the longest branch. Due to the structure of the algorithm, different branches of the tree may have varying depths, and the maximum possible number of leaf nodes is $2^{L-1}$. 

Let $K_l$, $1 \leq l \leq L$ be the number of subdomains at the $l$-th iteration of the approximation algorithm. This quantity corresponds to the total number of nodes at level $l$ of the tree, combined with the number of leaves from all previous levels. For ease of notation, we replace the Boolean vector used to label each subdomain at the $l$-th iteration of the algorithm with a scalar index $k$, $1 \leq k \leq K_l$. Given a parameter $\mu \in \mathcal{D}$, we identify the subdomain $\mathcal{D}_k \subset \mathcal{D}$ that contains $\mu$, and the nonlinear RB approximation at this iteration of the algorithm denoted by $u_{N,k}(\mu)$ is evaluated such that 
\begin{equation}
    a(u_{N,k}(\mu), v ; \mu)=f(v ; \mu), \quad \forall v \in V_{N,k},
\end{equation}
where $V_{N,k}$ is the linear RB space associated to the subdomain $\mathcal{D}_k$. It should be noted that, during the online stage, the final iteration $l=L$ is used and that $K=K_L$.

Furthermore, the nonlinear RB residual at the $l$-th iteration of the algorithm is given by
\begin{equation}
    r_{N,k}(v;\mu)=f(v;\mu)-a(u_{N,k}(\mu),v;\mu), \quad \forall v \in V_{\mathcal{N}}.
\end{equation}
The Riesz's representation, $R_{N,k}(\mu) \in V$, satisfies
\begin{equation}
    \big(R_{N,k}(\mu),v\big)_V=r_{N,k}(v;\mu),\quad \forall v \in V_{\mathcal{N}}.
\end{equation}
Then, the local RB error bound is given by 
\begin{equation}
    \eta_{N,k}(\mu)=\frac{\|R_{N,k}(\mu)\|}{\alpha_{LB}(\mu)}.
\end{equation}
Furthermore, the maximum local error estimator $\eta_{B_l}$ over a training set $E_l$ on the subdomain $\mathcal{D}_{B_l}$ is given by  
\begin{equation}
     \eta_{B_l}=\max\limits_{\mu\in E_{B_l}} \eta_{N,k}(\mu), 
\end{equation}
where $k$ is the scalar index associated with the subdomain $\mathcal{D}_{B_l}$. Finally, Lemma \eqref{lemma} remains applicable in this context with appropriate adjustments to the notation.

\begin{remark}
In practice, if for a given subdomain the targeted tolerance is achieved at a smaller number of snapshots than $N$, the linear RB space is built upon those snapshots and the branch is terminated. Therefore, the final subdomains may be of different sizes that are at most $N$.
\end{remark}

\begin{remark}
    We use the following two observations in our theoretical results. First, we observe that the given partitioning technique ensures that the largest subdomain in any partition contains the longest side in that partition.
    
    Second, the volume of a subdomain can be bounded below in terms of the length of its longest side. In particular, let $\mathcal{D}$ be a $d$-dimensional parameter domain with side lengths $h_j$, $1 \leq j \leq d$, and define
    \begin{equation}
        \beta=\frac{1}{2}\min_{1 \leq i , j \leq d} \frac{h_{i}}{h_{j}}.
    \end{equation}
    Furthermore, let $\delta_k$, $1 \leq k \leq K_l$, be the volume of each subdomain at the $l$-th iteration of the algorithm, and let $h_{k,j}$, $1 \leq k \leq K_l$, $1 \leq j \leq d$, denote the lengths of their sides. Without loss of generality, assume that $h_{k,1}$ is the length of the longest side. Then,
    \begin{align}
        \delta_k &= h_{k,1} \ h_{k,2} \cdots h_{k,d}\\
        &= h_{k,1} \ h_{k,1} \frac{h_{k,2}}{h_{k,1}} \cdots h_{k,1} \frac{h_{k,d}}{h_{k,1}}\\
        &\geq (\beta_k)^{d-1} (h_{k,1})^d, \quad \beta_k=\min_{1 \leq i , j \leq d} \frac{h_{k,i}}{h_{k,j}}.
    \end{align}
    In general,
    \begin{equation*}
        \delta_k \geq (\beta_k)^{d-1} (\hat{h}_k)^d, \quad \hat{h}_k= \max_{1\leq j\leq d} h_{k,j}.
    \end{equation*}
    Performing the split in the direction of the longest edge ensures that
    \begin{equation}
        \beta_k \geq \beta
    \end{equation}
    and
    \begin{equation}
        \delta_k \geq (\beta)^{d-1} (\hat{h}_k)^d.
        \label{beta}
    \end{equation}
    \label{rem geo}
\end{remark}

%%%%%%%%%%%%%%%%%%%%%%%%%%%%%%%%%%%%%%%%%%%%%%%%%%%%%%%%%%%%%%%%%%%%%%%%%%%%%%%%%%%%%%%%%%%%%%%%%%%%
\subsection{Convergence analysis}
The present analysis is influenced by the framework developed in \cite{cohen2015approximation} and employs it to study how the size of the parameter domain impacts the accuracy of the solution approximation.
We start by considering an elliptic parametric PDE of the form \eqref{para pde} with a parameter domain normalized to $\mathcal{D} = [-1,1]^d$. The following assumption is imposed on the parametric dependence of the solution.

\begin{assumption}
The parameter-to-solution map $\mu \mapsto u_{\mathcal{N}}(\mu)$ admits an extension to an open set $\mathcal{O} \subset \mathbb{C}^d$ containing the parameter domain $\mathcal{D}$ such that, for any $z \in \mathcal{O}$, the map $z \mapsto u_{\mathcal{N}}(z)$ is holomorphic in each variable $z_j$ with the uniform bound
\begin{equation}
    \sup_{z \in \mathcal{O}} \|u_{\mathcal{N}}(z)\|_V \leq C.
\end{equation}
\label{assump para}
\end{assumption}

Consequently, the unit polydisc
\begin{equation}
    \mathcal{P} := \left\{z = \left(z_j\right)_{j=1}^d \in \mathbb{C}^d : \left|z_j\right| \leq 1\right\} = \otimes_{j=1}^d \left\{z_j\in \mathbb{C} : \left|z_j\right| \leq 1\right\}
\end{equation}
is contained in $\mathcal{O}$. Since $\mathcal{O}$ is open, there exists an $\varepsilon>0$ such that the enlarged polydisc
\begin{equation}
    \mathcal{P}_{\varepsilon}:=\left\{z=\left(z_j\right)_{j=1}^d\in \mathbb{C}^d :\left|z_j\right| \leq 1+\varepsilon \right\}=\otimes_{j=1}^d \left\{z_j\in \mathbb{C} : \left|z_j\right| \leq 1+\varepsilon\right\}
\end{equation}
remains in $\mathcal{O}$.

We assume that the original parameter domain is partitioned into $K$ subdomains each of which has side lengths $h_k=(h_{k,1},h_{k,2},...,h_{k,d})$, $1 \leq k \leq K$. Let ${\mathcal{D}}_k\subset \mathcal{D}$ be any of the subdomains. It is clear that ${\mathcal{D}}_k\subset \mathcal{O}$. For a suitable choice of a vector $\rho_{k} = (\rho_{k,j})_{j=1}^d$, we define the associated polydisc as  
\begin{equation}
    \mathcal{P}_{\rho_k}:=\left\{z=\left(z_j\right)_{j=1}^d:\left|z_j\right| \leq \rho_{k,j} \right\}=\otimes_{j=1}^d\left\{\left|z_j\right| \leq \rho_{k,j} \right\}.
    \label{poly vec}
\end{equation}
We construct $\mathcal{P}_{\rho_k}$ such that it contains the subdomain $\mathcal{D}_k$ and is itself contained within the set $\mathcal{O}$. A possible choice for the components of $\rho_k$ that satisfies this condition is
\begin{equation}
    \rho_{k,j}=\frac{h_{k,j}}{2}+\varepsilon.
\end{equation}

Now, we assume that ${\mathcal{D}}_k$ is centered around $\hat{\mu}_k \in \mathcal{D}$. Then, for any $\mu \in \mathcal{D}_k$, where $\mathcal{D}_k=\prod_{j=1}^dI_{k,j}$, with $I_{k,j}=[\hat{\mu}_k-\frac{h_{k,j}}{2},\hat{\mu}_k+\frac{h_{k,j}}{2}]$, we introduce the following normalization
\begin{equation}
    \mu_{k,j}=\frac{2(\mu_j-\hat{\mu}_{k,j})}{h_{k,j}} \in [-1,1].
    \label{normalize}
\end{equation}
We extend this normalization to both the polydisc $\mathcal{P}{\rho_k}$ and the open set  $\mathcal{O}$. Specifically, the polydisc $\mathcal{P}{\rho_k}$ is mapped to a rescaled polydisc $\mathcal{P}{\hat{\rho}_k}$ such that 
\begin{equation}
    \hat\rho_{k,j}=1+\frac{2\varepsilon}{h_{k,j}},  \quad 1\leq j\leq d.
\end{equation}
Similarly, the open set $\mathcal{O}$ is transformed under this normalization to its rescaled counterpart, denoted by $\hat{\mathcal{O}}$. Accordingly, for any $\mu \in \mathcal{D}_k$, the solution $u_{\mathcal{N}}(\mu)$ can be expressed as
\begin{equation}
    u_{\mathcal{N}}(\mu)=u_{\mathcal{N},k} (\mu_k),
\end{equation}
where $u_{\mathcal{N},k}(\mu_k)$ solves the discrete parametric PDE in the normalized version of the subdomain $\mathcal{D}_k$.

%%%%%%%%%%%%%%%%%%%%%%%%%%%%%%%%%%%%%%%%%%%%%%%%%%%%%%%%%%%%%%%%%%%%%%%%%%%%%%%%%%%%%%%%%%%%%%%%%%%%%

\subsubsection{Convergence of $N$-term truncated polynomial expansion}
The goal of this section is to establish convergence rates of polynomial approximation of the parameter-to-solution map $\mu \mapsto u(\mu)$ over subdomains of the original parameter domain $\mathcal{D}$. These convergence rates are subsequently inherited by the Kolmogorov $N$-widths associated with the corresponding local solution manifolds. The approximation is constructed by truncating the Taylor series expansion of $u(\mu)$ to $N$ selected terms.

For any given subdomain $\mathcal{D}_k \subset \mathcal{D}$, $1 \leq k \leq K$, the polynomial approximation is based on the Taylor series expansion 
\begin{equation}
    u_{\mathcal{N}}(\mu)=u_{\mathcal{N},k}(\mu_k)=\sum_{\nu \in \mathbb{N}^d} t_{k,\nu} (\mu_k)^{\nu}, 
\end{equation}
where $t_{k,\nu}$ are the Taylor coefficients given by
\begin{equation}
    t_{k,\nu}=\frac{\partial^{\nu}u_{\mathcal{N},k}(0)}{\nu!}
\end{equation}
and
\begin{equation}
    (\mu_k)^{\nu}=\prod_{j=1}^d (\mu_{k,j})^{\nu_j}
\end{equation}
with $\mu_k$ defined by \eqref{normalize}. We begin by estimates on the coefficients of the Taylor expansion that is based on the holomorphic extension of the solution map.

We recall the Cauchy integral formula, which asserts that if $\varphi$ is a holomorphic function from $\mathbb{C}$ into a Banach space $V$, defined on a simply connected open subset $\mathcal{O} \subset \mathbb{C}$, and if $\Gamma$ is a closed, rectifiable path entirely contained within $\mathcal{O}$, then for any point $\tilde{z}$ inside the region bounded by $\Gamma$, the function $\varphi$ satisfies
\begin{equation}
    \varphi(\tilde{z}) = \frac{1}{2\pi i} \int_{\Gamma} \frac{\varphi(z)}{\tilde{z} - z} \, \mathrm{d}z,
    \label{cauchy}
\end{equation}
where the division represents scalar multiplication of the vector $\varphi(z) \in V$ by the complex inverse $(\tilde{z} - z)^{-1}$, and the curve $\Gamma$ is traversed in the positive (counterclockwise) direction~\cite{Hervé+1989}.

We know that $u_{\mathcal{N},k}$ is holomorphic in the set $\hat{\mathcal{O}}$, and that $\hat{\mathcal{O}}$ is an open neighborhood of the $d$-dimensional polydisc $\mathcal{P}_{\hat\rho_k}$. In addition, we have
\begin{equation}
    \sup_{z \in \hat {\mathcal{O}}} \|u_{\mathcal{N},k}(z)\|_V = \sup_{z \in \mathcal{O}} \|u_{\mathcal{N}}(z)\|_V  \leq C.
    \label{bound on poly}
\end{equation}
We may thus apply the Cauchy formula \eqref{cauchy} recursively in each variable $z_j$, and obtain for any $(\tilde{z}_1, \dots, \tilde{z}_d)$ in the interior of $\mathcal{P}_{\hat\rho_k}$ a representation of $u_{\mathcal{N},k}(\tilde{z}_1, \dots, \tilde{z}_d)$ as a multiple integral
\begin{equation}
u_{\mathcal{N},k}(\tilde{z}_1, \dots, \tilde{z}_d)
= (2\pi i)^{-d} \int_{|z_1| = \hat\rho_{k,1}} \cdots \int_{|z_d| = \hat\rho_{k,d}} 
\frac{u_{\mathcal{N},k}(z_1, \dots, z_d)}{(\tilde{z}_1 - z_1) \cdots (\tilde{z}_d - z_d)} \, \mathrm{d}z_1 \cdots \mathrm{d}z_d.
\end{equation}
Differentiating this expression yields
\begin{equation}
\frac{\partial^{|\nu|}}{\partial \tilde{z}_1^{\nu_1} \cdots \partial \tilde{z}_d^{\nu_d}} u_{\mathcal{N},k}(0, \dots, 0)
= \nu! (2\pi i)^{-d} \int_{|z_1| = \hat\rho_{k,1}} \cdots \int_{|z_d| = \hat\rho_{k,d}}
\frac{u_{\mathcal{N},k}(z_1, \dots, z_d)}{z_1^{\nu_1 + 1} \cdots z_d^{\nu_d + 1}} \, \mathrm{d}z_1 \cdots \mathrm{d}z_d.
\end{equation}
Therefore, using \eqref{bound on poly} , we obtain the estimate
\begin{equation}
\left\| \partial^{\nu} u_{\mathcal{N},k}(0) \right\|_V = \left\| \frac{\partial^{|\nu|} u_{\mathcal{N},k}}{\partial \tilde{z}_1^{\nu_1} \cdots \partial \tilde{z}_d^{\nu_d}}(0, \dots, 0) \right\|_V 
\leq C \nu! \prod_{j \leq d} \hat\rho_{k,j}^{-\nu_j}.
\end{equation}
It follows that
\begin{equation}
    \left\|t_{k,\nu}\right\|_V \leq C \hat\rho_k^{-\nu}=C \prod_{j=1}^d \hat\rho_{k,j}^{-\nu_j}, \quad \nu \in \mathbb{N}^d.
    \label{taylor bound}
\end{equation}
This estimate holds for any $\hat\rho_k$ such that the polydisc $\mathcal{P}_{\hat\rho_k}$ is contained in the open set $\hat{\mathcal{O}}$.

We denote by 
\begin{equation}
    h:=\max_{\substack{1 \leq k \leq K \\ 1 \leq j \leq d}}h_{k,j}
    \label{h}
\end{equation}
the maximum side length across all subdomains in the partition of the parameter domain. For each subdomain $\mathcal{D}_k$, we denote the associated solution manifold by
\begin{equation}
    \mathcal{M}_{\mathcal{N},k}=\left\{u_{\mathcal{N}}(\mu) \mid \mu \in \mathcal{D}_k\right\} \subset V, \quad 1 \leq k \leq K.
\end{equation}

We now exploit the estimate \eqref{taylor bound} to establish the following result concerning the Kolmogorov $N$-widths associated with the solution manifolds over the subdomains.

\begin{proposition} 
    Consider an elliptic parametric PDE of the form \eqref{para pde} with a $d$-dimensional parameter domain such that the parameter-to-solution map satisfies Assumption \ref{assump para}. Furthermore, let $\mathcal{D}$ be partitioned into $K$ tensor-product-structured subdomains ${\mathcal{D}_k}$,$1 \leq k \leq K$, and let $h$ denote the maximum side length of the partition as defined in \eqref{h}. Then there exists a constant $\hat{C} > 0$, independent of $h$, such that
    \begin{equation}
        \max_{1 \leq k \leq K} d_{N}(\mathcal{M}_{\mathcal{N},k}) \leq \hat C h^{N^{1/d}}.
    \label{hat C}
    \end{equation}
\end{proposition}

\begin{proof}
    We take $\rho_h=(\rho_{h,j})_{j=1}^d$ such that
    \begin{equation}
        \rho_{h,j}=\zeta_h :=1+\frac{2\varepsilon}{h}>1, \quad 1\leq j \leq d.
        \label{rho}
    \end{equation}
    It is clear that this choice of $\rho_h$ ensures that the polydisc $\mathcal{P}_{\rho_h}$ is contained in the open set $\hat{\mathcal{O}}$ for all subdomains $\mathcal{D}_k$. Let $\{\nu^{(n)}\}_{n \in \mathbb{N}}$ denote the multi-indices $\nu$ such that $\rho_h^{-\nu}=\zeta_h^{-|\nu|}$ are arranged in a non-increasing order. Let $\Lambda_N$ be a set containing the indices $\nu$ corresponding to the $N$ largest values of $\zeta_h^{-|\nu|}$. We consider the specific threshold
    \begin{equation}
        k_h= \hat{N} \lambda_{h}, 
        \label{kh}
    \end{equation}
    where
    \begin{equation}
        \hat{N}:=\left|\nu^{(N)}\right| +1, \quad \lambda_h:= \text{ln} (\zeta_h).
        \label{lambda def}
    \end{equation}
    Then, we have 
    \begin{equation}
        \Lambda_N\subseteq\left\{\nu \in \mathbb{N}^d: \zeta_h^{-|\nu|} > e^{-k_h}\right\}=:S_{k_h}.
    \end{equation}
    Equivalently, 
    \begin{equation}
        S_{k_h}:=\left\{\nu \in \mathbb{N}^d: \left|\nu\right| \lambda_h < k_h\right\}.
        \label{k}
    \end{equation}
    Due to ties in the values of $\zeta_h^{-|\nu|}$, the set $S_{k_h}$ may contain more than $N$ indices. In this case, the set $\Lambda_N$ s obtained by discarding indices corresponding to the smallest values of $\zeta_h^{-|\nu|}$ in $S_{k_h}$ so that $|\Lambda_N| = N$.
    
    The cardinality of $\Lambda_{N}$ is bounded by the number of all partial derivatives up to order $\hat{N}-1$ which is the maximum number of derivatives that could appear in the Taylor expansion if we consider the $N$ largest values of $\zeta_h^{-|\nu|}$. Then,
    \begin{align}
        N&\leq \sum_{n=0}^{\hat{N}-1} {n+d-1 \choose d-1} = \frac{\left(\hat{N}-1+d\right)!}{\left(\hat{N}-1\right)! \ d!}\\
        &= \frac{\left(\hat{N}-1+d\right) \left(\hat{N}-1+(d-1)\right) \cdots \left(\hat{N}\right)}{d!}\\
        &=\left(\frac{\hat{N}-1}{d} +1\right)\left(\frac{\hat{N}-1}{d-1} +1 \right) \cdots (\hat{N})\\
        &\leq (\hat{N})^d\\
        &=\left( \frac{k_h}{\lambda_h}\right)^d
        \label{est N}
    \end{align}
    
    In general, sets $\Lambda_n$ of arbitrary size $n$ corresponding to a threshold $k$ consist of at most all integer lattice points inside the simplex bounded by the coordinate hyperplanes together with the hyperplane $\lambda_{h} \sum_{j=1}^d t_j = k$. These sets are downward closed and their cardinality is bounded by the volume of the following continuous simplex
    \begin{equation}
        T_{k}:=\left\{\left(t_1, \ldots, t_d\right) \in \mathbb{R}^d: t_j \geq-1, j=1, \ldots, d, \text { and } \sum_{j=1}^d t_j \leq \frac{k}{\lambda_{h}} \right\}.
    \end{equation}
    Then,
    \begin{align}
        |\Lambda_n|=|S_{k}|\leq\left|T_k\right|&=\frac{1}{d!} \left(\frac{k+d \lambda_h}{\lambda_h}\right)^d.
        \label{set size}
    \end{align}
    
    For any $\mu \in \mathcal{D}$ there exists a subdomain $\mathcal{D}_k \subset \mathcal{D}$ such that $\mu \in \mathcal{D}_k$. Therefore, the approximation error when keeping only the $N$ terms with indices in $\Lambda_N$ is given by 
    \begin{align}
          \left\|u_{\mathcal{N}}(\mu)-\sum_{\nu \in \Lambda_N} t_{k,\nu} \mu_k^\nu \right\|_V 
        & \leq \sum_{\nu \notin S_{k_h}} \left\|t_{k,\nu} \right\|_V \\
        & \leq C\sum_{\nu \notin S_{k_h}} \zeta_h^{-|\nu|} \\
        & \leq C \sum_{l \geq k_h } e^{-l} \left|\left\{\nu: e^{-l-1} < \zeta_h^{-|\nu|} \leq e^{-l}\right\}\right|\\ %because we can only deal with (e^{-l-1} < \zeta_h^{-|\nu|} ) through S_k
        & \leq C \sum_{l \geq k_h } e^{-l} \left|S_{l+1}\right|.
    \end{align}
    Using the estimate \eqref{set size}, we obtain
    \begin{equation}
          \left\|u_{\mathcal{N}}(\mu)-\sum_{\nu \in \Lambda_N} t_{k,\nu} \mu_k^\nu \right\|_V  \leq \frac{C}{d! \ (\lambda_h)^d} \sum_{l \geq k_h } e^{-l} \left(l+1+ d\lambda_h\right)^d.
        \label{err_1}
    \end{equation}
    From \eqref{kh}, estimate \eqref{err_1} becomes
    \begin{align}
          \left\|u_{\mathcal{N}}(\mu)-\sum_{\nu \in \Lambda_N} t_{k,\nu} \mu_k^\nu \right\|_V &  \leq \frac{C \ \hat{N}^d}{d! \  (k_h)^d} \sum_{l \geq k_h } e^{-l} \left(l+1+\hat{N}^{-1}d k_h\right)^d\\
        & \leq \frac{C \ \hat{N}^d}{d!} \frac{\Phi\left(e^{-1},-d,(\hat{N}^{-1}d+1)k_h+1\right)}{(k_h)^d} e^{-k_h},
    \end{align}
    where $\Phi$ is the Hurwitz Lerch transcendent which converges for these values and gives a polynomial of degree $d$ in $k_h$. Consequently, the ratio $\frac{\Phi}{(k_h)^d}$ is uniformly bounded as $h\rightarrow 0$ $\left(k_h \rightarrow \infty\right)$. As a result, there exisits a constant $C_1$ such that
    \begin{align}
          \left\|u_{\mathcal{N}}(\mu)-\sum_{\nu \in \Lambda_N} t_{k,\nu} \mu_k^\nu \right\|_V &\leq C_1 e^{-k_h}\\
        &\leq C_1 e^{-\lambda_hN^{1/d}},
    \end{align}
    where $C_1$ is a constant that depends on $d$ and $N$, and the second inequality comes from estimate \eqref{est N}. From the definitions \eqref{rho} and \eqref{lambda def}, it can be inferred that 
    \begin{align}
        e^{-\lambda_h}&=\frac{1}{\zeta_h}= \frac{h}{h +2\varepsilon}\leq \frac{h}{2\varepsilon}.
    \end{align} 
    Therefore, we obtain the error estimate
    \begin{equation}
          \max_{1 \leq k \leq K}\sup_{\mu \in \mathcal{D}_k} \left\|u_{\mathcal{N}}(\mu)-\sum_{\nu \in \Lambda_N} t_{k,\nu} \mu_k^\nu \right\|_V \leq \hat C h^{N^{1/d}}.
    \end{equation}
    As a result, the Kolmogorov $N$-width inherits the same convergence rate which completes the proof
    
\end{proof}

%%%%%%%%%%%%%%%%%%%%%%%%%%%%%%%%%%%%%%%%%%%%%%%%%%%%%%%%%%%%%%%%%%%%%%%%%%%%%%%%%%%%%%%%%%%%%%%%%%%%%
\subsubsection{Convergence theory of the nonlinear RB method}
We recall that $L$ denotes the total number of iterations of the approximation algorithm. At each iteration $l$, where $1 \leq l \leq L $, the domain is partitioned into  $K_l$  subdomains. For each subdomain $k$, with $ 1 \leq k \leq K_l $,  $\mathcal{M}_{\mathcal{N},k}$ denotes the corresponding solution manifold at the $l$-th iteration.

Given $1 \leq l \leq L$ and $1 \leq k \leq K_l$, we denote by 
\begin{equation}
    \hat{\sigma}_l(\mathcal{M}_{\mathcal{N}}):= \max_{1 \leq k \leq K_l} \sigma_N(\mathcal{M}_{\mathcal{N},k})
    \label{def sigma}
\end{equation}
the maximum greedy approximation error at at the $l$-th iteration of the algorithm. Due to Céa's lemma, we have 
\begin{equation}
    \max_{1 \leq k \leq K_l}\sup_{\mu \in \mathcal{D}_k} \left\|u_{\mathcal{N}}(\mu)-u_{N,k}(\mu)\right\|_V \leq  \frac{\overline{\gamma}}{\underline{\alpha}} \hat{\sigma}_l(\mathcal{M}_{\mathcal{N}}),
    \label{cea}
\end{equation}
where
\begin{align}
    \underline{\alpha} & :=\inf _{\mu \in \mathcal{D}} \alpha_{LB}(\mu), \\
    \overline{\gamma} & :=\sup _{\mu \in \mathcal{D}} \gamma_{UB}(\mu),
\end{align}
are assumed to be available. Furthermore, let $\delta_{k}$ denote the volume of the subdomain $\mathcal{D}_k$. We assume that the approximation algorithm satisfies the following assumption:
\begin{assumption}
    \label{quasi}
    The parameter domain partition resulting from the approximation algorithm is always quasi-uniform, that is there exists $\xi > 0$ such that for all $\epsilon>0$, we have
    \begin{equation}
        \underline{\delta}_L \geq \xi \overline{\delta}_L,
        \end{equation}
    where
    \begin{equation}
        \underline{\delta}_L := \min _{1 \leq k \leq K_L} \delta_{k}, \quad \overline{\delta}_L := \max _{1 \leq k \leq K_L} \delta_{k}.
    \end{equation}
\end{assumption}
This assumption can be readily enforced during the design of the algorithm by subdividing large subdomains that violate the quasi-uniformity condition. However, in many practical scenarios, it arises naturally, particularly when the variation of the solution with respect to the parameter is roughly uniform across the entire parameter domain. It is important to note that this assumption is required only for the theoretical analysis and is not enforced in practice.

Now, we are in position to state our main result, which characterizes the effect of varying the tolerance $\epsilon$ on the number of subdomains $K$ resulting from the nonlinear RB approximation for fixed values of $N$ and $d$.
\begin{theorem} 
    Consider an elliptic parametric PDE of the form \eqref{para pde} with a $d$-dimensional parameter domain such that the parameter-to-solution map satisfies Assumption \ref{assump para}. 
    Let $N$ denote the dimension of the RB space used in the nonlinear RB approximation and for any prescribed tolerance $\epsilon>0$, let $K(\epsilon)$ be the number of resulting subdomains. If Assumption \ref{quasi} holds, then there exists a constant $C > 0$, independent of $\epsilon$, such that
    \begin{equation}
        K(\epsilon) \leq \max\left\{1, \frac{C}{\epsilon^{d/N^{1/d}}}\right\}.
    \end{equation}
    \label{th conv}
\end{theorem}

\begin{proof}
    The nonlinear RB approximation algorithm terminates if the error estimator over each subdomain satisfies the required tolerance $\epsilon$. If $L=1$ (hence $K=1$) the proof is complete. Otherwise, for $L>1$, $ 1\leq l \leq L-1$ and $1 \leq k \leq K_l$, we have
    \begin{equation}
        \epsilon < \hat\eta_l,
        \label{eps=1}
    \end{equation}
    where
    \begin{equation}
        \hat\eta_l:=\max\{\eta_{N,k}(\mu) : \mu \in \mathcal{D}_k, 1\leq k \leq K_l \}
    \end{equation}
    is the maximum a posteriori error estimate at the $l$-th iteration. We define 
    \begin{equation}
        h_l:=\max_{\substack{1 \leq k \leq K_l \\ 1 \leq j \leq d}}h_{k,j},
    \end{equation}
    and use the upper bound in \eqref{est} on the error estimator to get 
    \begin{align}
    \hat\eta_l &\leq \frac{\overline{\gamma}}{\underline{\alpha}} \max_{1 \leq k \leq K_l}\sup_{\mu \in \mathcal{D}_k} \left\|u_{\mathcal{N}}(\mu)-u_{N,k}(\mu)\right\|_V\\
    &\leq \frac{\overline{\gamma}^2}{\underline{\alpha}^2} \hat{\sigma}_l(\mathcal{M})\\
    &\leq \hat C \frac{\kappa ^{N+1}\overline{\gamma}^2}{{\sqrt{3}\underline{\alpha}^2}} (h_l)^{N^{1/d}} ,
    \label{eta_1}
    \end{align}
    where the second inequality comes from \eqref{cea} and the third inequality comes from \eqref{th comp} and \eqref{hat C}. From \eqref{eps=1}, we deduce that 
    \begin{equation}
        h_l > \left(\frac{\sqrt{3}\underline{\alpha}^2 \epsilon}{\hat C \kappa ^{N+1}\overline{\gamma}^2}\right)^{1/N^{1/d}}.
        \label{euc}
    \end{equation}

    Since the domain partition is quasi-uniform, we have
    \begin{equation}
        \underline{\delta}_L \geq \xi \overline{\delta}_L.
    \end{equation}
    Given that the intermediate subdomains are equally split, we obtain
    \begin{equation}
        \underline{\delta}_L \geq \frac{\xi}{2} \overline{\delta}_{L-1}.
    \end{equation}
    Using estimate \eqref{beta}, we get
    \begin{equation}
        \underline{\delta}_L \geq \frac{\xi \beta^{d-1}}{2} (h_{L-1})^d.
    \end{equation}
    Thus by (\ref{euc}), we have 
    \begin{equation}
        \underline{\delta}_L  > \frac{\xi \beta^{d-1} }{2} \left(\frac{\sqrt{3}\underline{\alpha}^2 \epsilon}{\hat C \kappa ^{N+1}\overline{\gamma}^2}\right)^{d/N^{1/d}}.
        \label{lower=1}
    \end{equation}
    
    Let $|\mathcal{D}|$ denote the volume of the parameter domain, then $K\underline{\delta}_L\leq |\mathcal{D}|$. Let 
    \begin{equation}
        C:=\frac{2|\mathcal{D}|}{\xi \beta^{d-1} } \left(\frac{\hat C \kappa ^{N+1}\overline{\gamma}^2}{\sqrt{3}\underline{\alpha}^2}\right)^{d/N^{1/d}},
        \label{const=1}
    \end{equation}
    and for the sake of contradiction, assume that $K>C/\epsilon^{d/N^{1/d}}$. From (\ref{const=1}) and (\ref{lower=1}), it follows that 
    \begin{equation}
        K\underline{\delta}_L>\frac{C}{\epsilon^{d/N^{1/d}}} \underline{\delta}_L > |\mathcal{D}|,
    \end{equation}
    which leads to a contradiction and completes the proof.
\end{proof}

%%%%%%%%%%%%%%%%%%%%%%%%%%%%%%%%%%%%%%%%%%%%%%%%%%%%%%%%%%%%%%%%%%%%%%%%%%%%%%%%%%%%%%%%%%%%%%%%%%%%5
\subsubsection{A model problem}
A frequently encountered model problem in the context of model order reduction is the linear elliptic partial differential equation
\begin{equation}
    \begin{aligned}
    -\operatorname{div} \big(a(\mu) \nabla u(\mu) \big)=f  \text { in } \Omega, \\
    u=0  \text { on } \partial \Omega,
    \label{the pde}
\end{aligned}
\end{equation}
where $f$ is a real-valued function, and the diffusion coefficient $a$ depends on a parameter 
$\mu$. The associated variational formulation reads: Given any $\mu \in \mathcal{D}$, find $u(\mu) \in V=H_0^1(\Omega)$ such that
\begin{equation}
    \int_\Omega a(\mu) \nabla u(\mu) \cdot \nabla v=\langle f, v\rangle, \quad v \in V,
    \label{var model}
\end{equation}
where $\langle \cdot, \cdot\rangle$ is the duality pairing between $V'$ and $V$. The diffusion coefficient $a$ is assumed to be affine in the parameter $\mu$, taking the form
\begin{equation}
   a:=a(\mu):= \bar{a}+\sum_{j=1}^d \mu_j \psi_j,
\end{equation}
where $\bar{a}\in L^{\infty}(\Omega)$ and  $\psi_j \in L^{\infty}(\Omega), 1\leq j \leq d$. According to the Lax–Milgram lemma, this problem is well-posed if $f \in V'$ and there exists a constant $r > 0$ such that the diffusion coefficient $a$ satisfies:
\begin{equation}
   \bar{a}(x)+\sum_{j=1}^d \mu_j \psi_j(x) \geq r, \quad  x\in \Omega,\ \mu \in \mathcal{D}.
    \label{UEA1}
\end{equation}
Relation \eqref{UEA1} is referred to as the uniform ellipticity assumption of constant $r$, or UEA$(r)$. Therefore, if the UEA(r) holds, the map $a \mapsto u(a)$ is well-defined over 
\begin{equation}
    a(\mathcal{D}):=\left\{a=\bar{a}+\sum_{j=1}^d \mu_j \psi_j: \mu \in \mathcal{D}\right\},
\end{equation}
and thus, the map $\mu \mapsto u(\mu)$ is well defined from $\mathcal{D}$ to $V$. Here, we adopt the slight abuse of notation $ u(\mu):=u(a(\mu))$.

In the case where $a$ is complex valued, we consider $V$ as a space of complex-valued functions. The variational formulation remains the same as in \eqref{var model}, with the integral on the left understood as the standard Hilbertian inner product, and $\langle f, v\rangle$ representing the anti-duality pairing between the complex spaces $V'$ and $V$. By the complex version of Lax-Milgram lemma, this problem admits a unique solution $u$ provided that
\begin{equation}
    \mathfrak{R}(a) \geq r, \quad x\in \Omega, \ \mu \in \mathcal{D},
\end{equation}
for some $r>0$. In that case, the solution satisfies the stability estimate 
\begin{equation}
        \left\|u\right\|_V  \leq \frac{\|f\|_{V'}}{r}.
\end{equation}
As a result, the map $a \mapsto u(a)$ corresponding to Problem \eqref{the pde} can be extended to the complex domain
\begin{equation}
    \mathcal{U}_r:=\{a \in X: \mathfrak{R}(a) \geq r\},
\end{equation}
with the uniform bound
\begin{equation}
    \sup \left\{\|u(a)\|_V:a \in \mathcal{U}_r\right\} \leq \frac{\|f\|_{V'}}{r} .
\end{equation}
Therefore, this extension is defined on the open set $\mathcal{U} :=\cup_{r>0} \mathcal{U}_r .$

For any $a\in X$, the sesquilinear form of the Problem \eqref{the pde} induces a bounded linear operator $\mathcal{B}(a) : V \mapsto W$, that is $\mathcal{B}(a): v \mapsto -\text{div}(a \nabla v)$, and the solution of the problem is $u(a):=\mathcal{B}(a)^{-1} f$. To see that the map $a \mapsto u(a)$ is holomorphic, we decompose it as follows.
\begin{equation}
    a \mapsto \mathcal{B}(a) \mapsto \mathcal{B}(a)^{-1} \mapsto \mathcal{B}(a)^{-1} f=u(a),
\end{equation}
where the first and third maps are continuous linear and therefore holomorphic. The second map is the operator inversion, which is holomorphic at any invertible $\mathcal{B} \in \mathcal{L}(V,{V'})$.

It is shown in \cite{cohen2015approximation} that if UEA($r$) holds, then there exists $\varepsilon>0$ such that for any vector $\tau=(\tau_{j})_{j=1}^d$, with each $\tau_j \geq 1$ and satisfying
\begin{equation}
    \sum_{j=1}^d (\tau_{j}-1) \|\psi_j\|_X \leq \varepsilon,
    \label{cons2}
\end{equation}
there exists an open set $\mathcal{O}$ containing the polydisc $\mathcal{P}_{\tau}$ (as defined in \eqref{poly vec}) for which the map $\mu \mapsto u(\mu)$ admits a holomorphic extension with the uniform bound
\begin{equation}
    \sup_{z \in \mathcal{O}} \|u_{\mathcal{N}}(z)\|_V \leq \frac{\|f\|_{V'}}{t}
\end{equation}
for some $t<r$. This result effectively ensures that Assumption \ref{assump para} holds. A possible choice for $\tau_j$ satisfying \eqref{cons2} is
\begin{equation}
    \tau_j=1+\frac{\varepsilon}{d\left\|\psi_j\right\|_{L^{\infty}(\Omega)}}, \quad 1\leq j \leq d.
\end{equation}

Let the parameter domain be partitioned into $K$ subdomains, each characterized by side lengths  $h_k=(h_{k,1},h_{k,2},...,h_{k,d})$, $1 \leq k \leq K$. A choice of $\rho_{k}=(\rho_{k,j})_{j=1}^d$ ensuring that the polydisc $\mathcal{P}_{\rho_k}$ contains the subdomain $\mathcal{D}_k$ and is contained in $\mathcal{O}$, is given by
\begin{equation}
    \rho_{k,j}=\frac{h_{k,j}}{2}+\frac{\varepsilon}{d\left\|\psi_j\right\|_{L^{\infty}(\Omega)}}, \quad 1\leq j \leq d.
\end{equation}
Under normalization \eqref{normalize}, $\rho_{k,j}$ becomes 
\begin{equation}
    \hat \rho_{k,j}=1+\frac{2\varepsilon}{dh_{k,j}\left\|\psi_j\right\|_{L^{\infty}(\Omega)}}, \quad 1\leq j \leq d.
\end{equation}
Accordingly, we may choose $\zeta_h$ in \eqref{rho} as
\begin{equation}
    \zeta_h=\min_{1\leq j \leq d}\left\{1+\frac{2\varepsilon}{hd\left\|\psi_j\right\|_{L^{\infty}(\Omega)}}\right\},
\end{equation}
where $h$ is defined in \eqref{h}.

%%%%%%%%%%%%%%%%%%%%%%%%%%%%%%%%%%%%%%%%%%%%%%%%%%%%%%%%%%%%%%%%%%%%%%%%%%%%%%%%%%%%%%%%%%%%%%%%%%%%
\section{Numerical results}
We now present numerical experiments to validate the theoretical convergence result established in Theorem~\ref{th conv}. We consider two parametric PDEs: a diffusion equation and a convection–diffusion equation. For each problem, we first apply the standard linear RB method, followed by the proposed nonlinear approach. We then compare the effectiveness of our method to the one presented in \cite{eftang2010hp}.
\begin{exmp}
\label{model prob}
For our first numerical example, we consider the model problem \eqref{the pde}, where $V=H_0^1(\Omega)$, and the physical domain is defined as
\begin{equation*}
    \Omega = \{(x, y) \in \mathbb{R}^2 : 0< x < 1, \ 0 < y < 1 \}
\end{equation*}
with boundary $\partial \Omega$. We take the right hand side as $f=1$. We let $\mu=(\mu_1,\mu_2) \in \mathcal{D}$, where $\mathcal{D}=[-1,1]^2$ (corresponding to the case $d=2$) and define the diffusion coefficient by
\begin{equation}
    a=1+\frac{\cos(2\pi x)+\cos(2\pi y)}{2\alpha\pi^2}\mu_1+\frac{\cos(4\pi x)+\cos(4\pi y)}{8\pi^2}\mu_2,
\end{equation}
where $\alpha =0.105$ ensuring the coercivity of the problem. The varitional formulation of the problem is given in \eqref{var model}.
\end{exmp}
Next, we introduce a standard finite element space $V_{\mathcal{N}}\subset V$ consisting of $3806$ linear triangular elements. The corresponding discrete problem reads: Given any $\mu \in \mathcal{D}$, find $u_{\mathcal{N}}(\mu) \in V_{\mathcal{N}}$ such that
\begin{equation}
    \int_\Omega a(\mu) \nabla u_{\mathcal{N}}(\mu) \cdot \nabla v=\int_\Omega v, \quad v \in V_{\mathcal{N}}.
    \label{fe model}
\end{equation}

In Figure~\ref{snaps_model}, high-fidelity solutions are shown for two different parameter values, illustrating how the structure of the solution changes as the parameters vary.
\begin{figure}[h]
    \centering
    \begin{subfigure}[h]{.48\textwidth}
    \includegraphics[width=.98\textwidth,height=.8\textwidth]{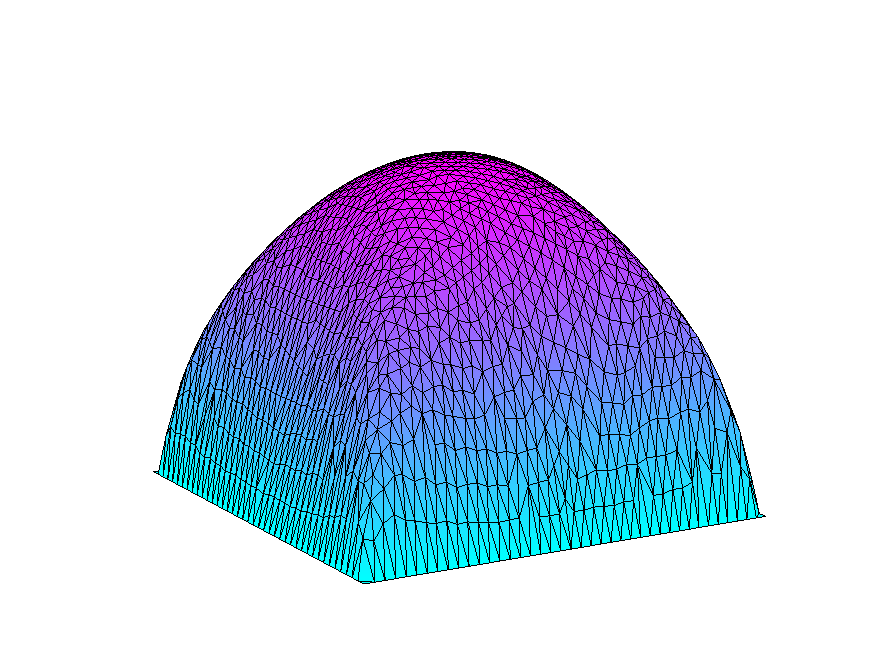}
    \caption{}
    \end{subfigure}
    \hfill
    \begin{subfigure}[h]{.48\textwidth}
    \includegraphics[width=.98\textwidth,height=.8\textwidth]{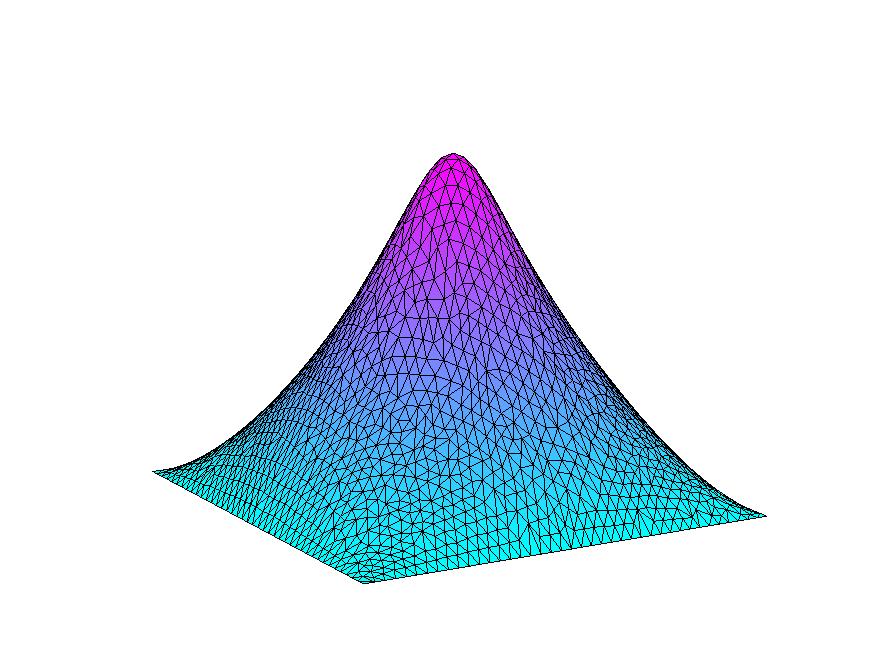}
    \caption{}
    \end{subfigure}
    \caption{Solutions of Problem \eqref{fe model} for different parameter values: (a) $\mu=(-1,1)$, and (b) $\mu=(1,1)$.}
    \label{snaps_model}
\end{figure}

We introduce a uniformly distributed random training set $E \subset \mathcal{D}$ consisting of $10^4$ parameters. Figure~\ref{greed_selec} shows the parameters chosen by the greedy algorithm. Most of these parameters have $\mu_1$ values close to $-1$ or $1$, where the problem approaches noncoerciveness. In Figure~\ref{linear_model}, we plot the maximum $V$-norm error bound, $\eta_{\max} = \max_{\mu \in E} \eta_N(\mu)$, over the corresponding training set, as a function of $N$ showing the relatively slow convergence for the chosen value of $\alpha$.

\begin{figure}[h]
    \centering
    \begin{subfigure}[h]{.48\textwidth}
    \includegraphics[width=.98\textwidth,height=.8\textwidth]{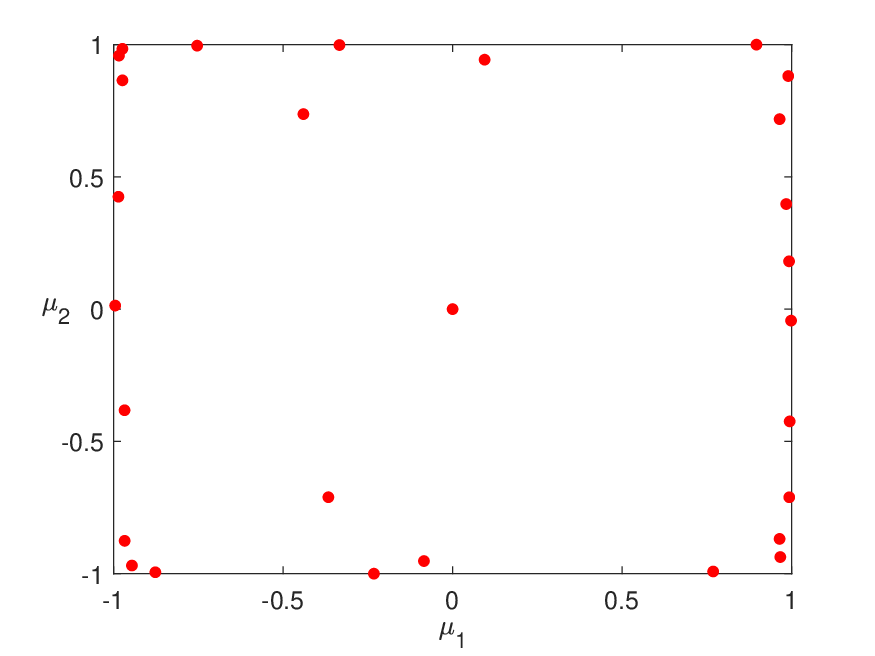}
    \caption{}
    \label{greed_selec}
    \end{subfigure}
    %\hfill
    \begin{subfigure}[h]{.48\textwidth}
    \includegraphics[width=.98\textwidth,height=.8\textwidth]{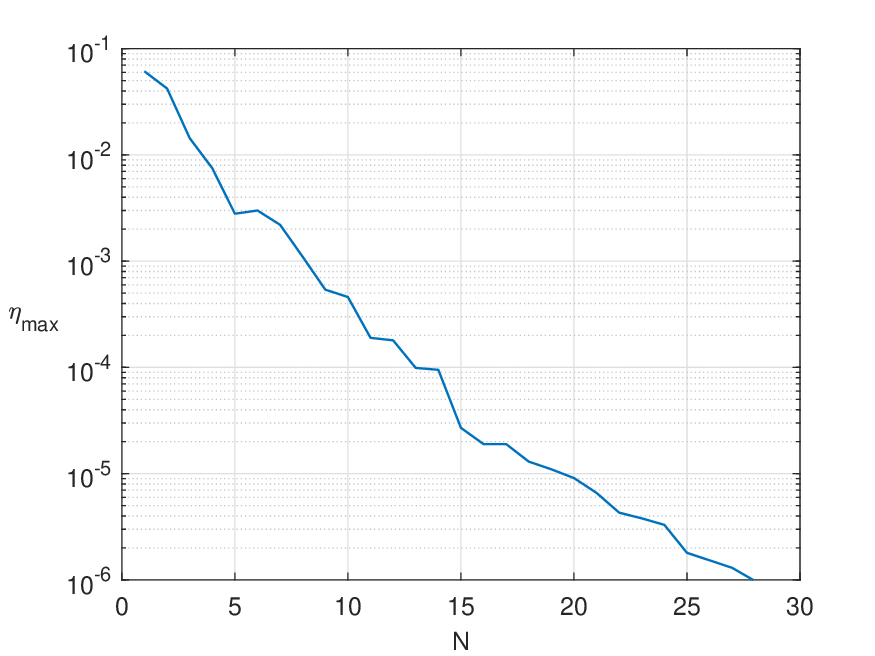}
    \caption{}
    \label{linear_model}
    \end{subfigure}
    \caption{(a) Greedy parameter selection. (b)Linear RB convergence results.}
\end{figure} 

We now present the convergence results for our proposed algorithm. We begin by specifying the RB space dimension $N$, the size of the training set, and the initial parameter. The proposed algorithm is executed for various values of the tolerance $\epsilon$, and we plot the number of subdomains $K$ as a function of $\epsilon$.

We compare the performance of our algorithm, referred to as Algorithm $1$, to the algorithm introduced in \cite{eftang2010hp}, referred to as Algorithm $2$. The algorithm in \cite{eftang2010hp} employs a different domain-splitting strategy, specifically based on a proximity function. If the maximum error estimator within a subdomain exceeds the tolerance, the subdomain is split into two parts. One subdomain retains the anchor parameter of the original subdomain, while the anchor parameter of the second is selected via the greedy algorithm. The partitioning is then determined based on proximity to the respective anchor points. Here, we adopt the Euclidean distance as our proximity function, consistent with the numerical results provided in \cite{eftang2010hp}.

We initialize the algorithm with the parameter value $(0,0)$ and set the size of the randomly sampled training set to $10^3$. In Figure~\ref{comp_model}, we present a comparison of the two algorithms by plotting the number of subdomains $K$ as a function of the tolerance $\epsilon$, considering three different cases: $N = 1$, $2$, and $4$. The observed convergence rates are in good agreement with the theoretical predictions established in Theorem~\ref{th conv} for the first two cases. For $N=4$, a slight superconvergence is observed. This could be attributed to the problem structure, which drives the greedy algorithm to select parameters clustered near the boundaries of the parameter domain, as shown in Figure~\ref{greed_selec}, and leads to a non–quasi-uniform domain partition, as illustrated in Figure~\ref{dom_partition}. Notably, both algorithms exhibit comparable performance in terms of accuracy.

\begin{figure}[h]
    \centering
    \includegraphics[width=.6\textwidth,height=.5\textwidth]{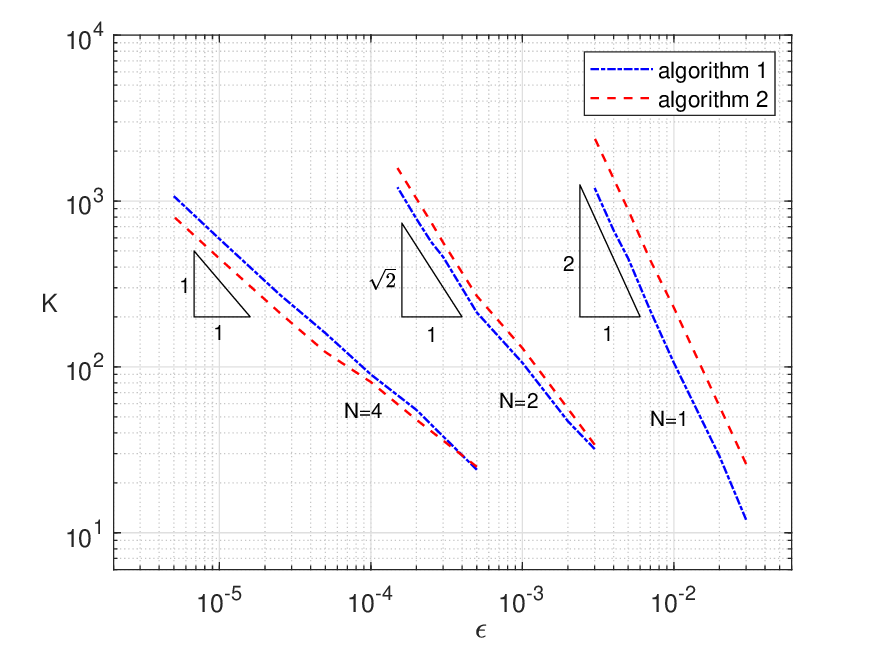}
    \caption{Convergence results for the two algorithms at $N=1,2,$ and $4$.}
    \label{comp_model}
\end{figure}

In Figure~\ref{dom_partition}, we present the partition of the parameter domain $\mathcal{D}$ for the two algorithms in the case $N = 4$ and $\epsilon = 5\times10^{-5}$. Clearly, the partition resulting from the proposed algorithm is much simpler than that produced by Algorithm 2. Consequently, the online storage required for the parameter domain partition is significantly smaller than that of the other algorithm.

\begin{figure}[h]
    \centering
    \begin{subfigure}[h]{.48\textwidth}
    \includegraphics[width=\textwidth,height=.8\textwidth ]{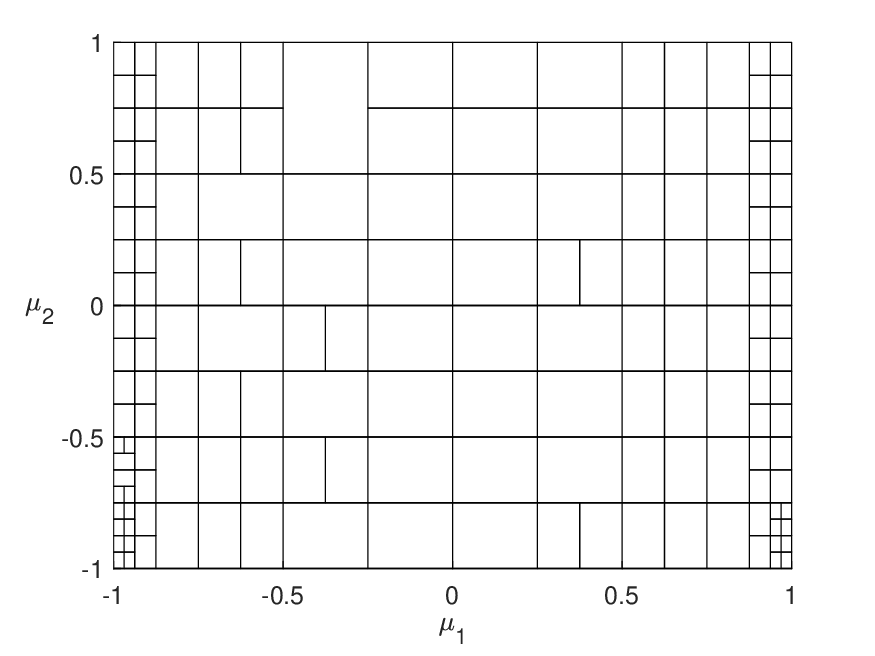}
    \caption{}
    \end{subfigure}
    %\hfill
    \begin{subfigure}[h]{.48\textwidth}
    \includegraphics[width=\textwidth,height=.8\textwidth ]{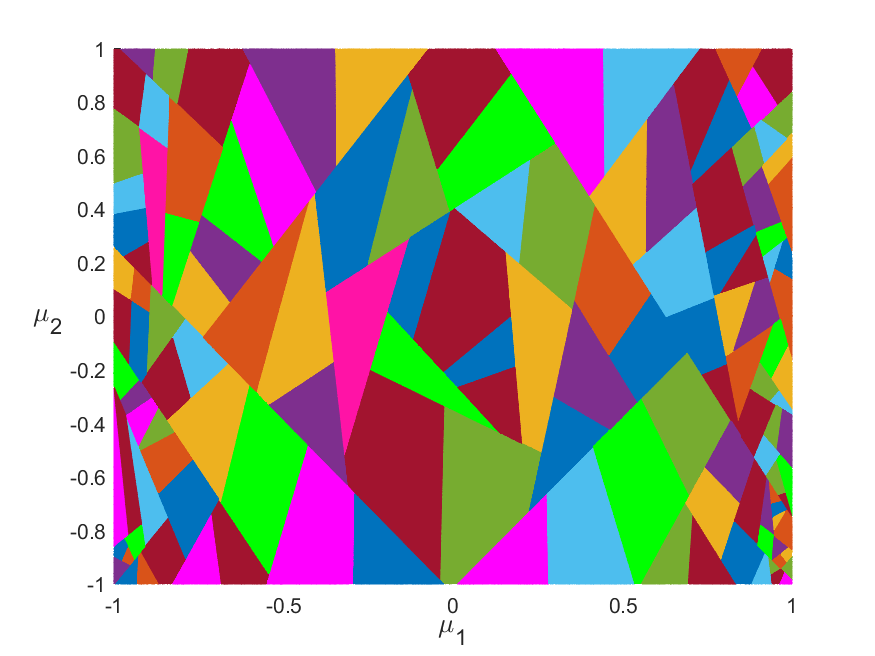}
    \caption{}
    \end{subfigure}
    \caption{Final parameter domain partition when $N=4$ and $\epsilon=5\times10^{-5}$ for (a) the proposed algorithm (b) the algorithm from \cite{eftang2010hp}.}
    \label{dom_partition}
\end{figure}

%%%%%%%%%%%%%%%%%%%%%%%%%%%%%%%%%%%%%%%%%%%%%%%%%%%%%%%%%%%%%%%%%%%%%%%%%%%%%%%%%%%%%%%%%%%%%%%%%%%%%
\begin{exmp}
\label{transport}
We consider the steady convection-diffusion model problem from \cite{eftang2010hp}. This problem is parametrized by the angle and magnitude of a prescribed velocity field. Let $\mu = (\mu_1, \mu_2)$ and define the velocity field as:
\begin{equation}
V(\mu) = [\mu_2 \cos \mu_1, \mu_2 \sin \mu_1]^T.
\end{equation}
The governing equations for the field variable $u(\mu)$ are:
\begin{equation}
-\Delta u(\mu) + V(\mu) \cdot \nabla u(\mu) = 10 \quad \text{in} \, \Omega,
\end{equation}
\begin{equation}
u(\mu) = 0 \quad \text{on} \, \partial \Omega,
\end{equation}
where the physical domain is defined as $\Omega = \{(x, y) \in \mathbb{R}^2 : x^2 + y^2 \leq 2\}$, and $\partial \Omega$ is the boundary of $\Omega$.
The function space associated with the given boundary condition is given by $ V := H_0^1(\Omega)$.

The weak parametrized formulation then reads: Given any $\mu \in \mathcal{D}$, find $u \in V$ such that
\begin{equation}
a(u(\mu), v; \mu) = f(v), \quad \forall v \in V,
\end{equation}
with 
\begin{align}
a(w, v; \mu) &= \int_{\Omega} \nabla w \cdot \nabla v   + \int_{\Omega} (V(\mu) \cdot \nabla w) v\\
&= \int_{\Omega} \nabla w \cdot \nabla v + \mu_2 \cos \mu_1 \int_{\Omega} \frac{\partial w}{\partial x} v + \mu_2 \sin \mu_1 \int_{\Omega} \frac{\partial w}{\partial y} v ,
\end{align}
and 
\begin{equation}
f(v) = f(v; \mu) = 10 \int_{\Omega} v ,
\end{equation}
for all $w, v \in V$.
\end{exmp}

Next, we introduce a standard finite element space $V_{\mathcal{N}}\subset V$ consisting of linear elements. The underlying mesh consists of $3689$ linear elements. Following the Galerkin approach, we obtain the following discrete problem: Given any $\mu \in \mathcal{D}$, find $u_{\mathcal{N}}(\mu) \in V_{\mathcal{N}}$ such that
\begin{equation}
a(u_{\mathcal{N}}(\mu), v; \mu) = f(v), \quad \forall v \in V_{\mathcal{N}}.
\label{FE prob}
\end{equation}

\begin{figure}[h]
    \centering
    \begin{subfigure}[h]{.32\textwidth}
    \includegraphics[width=\textwidth,height=\textwidth]{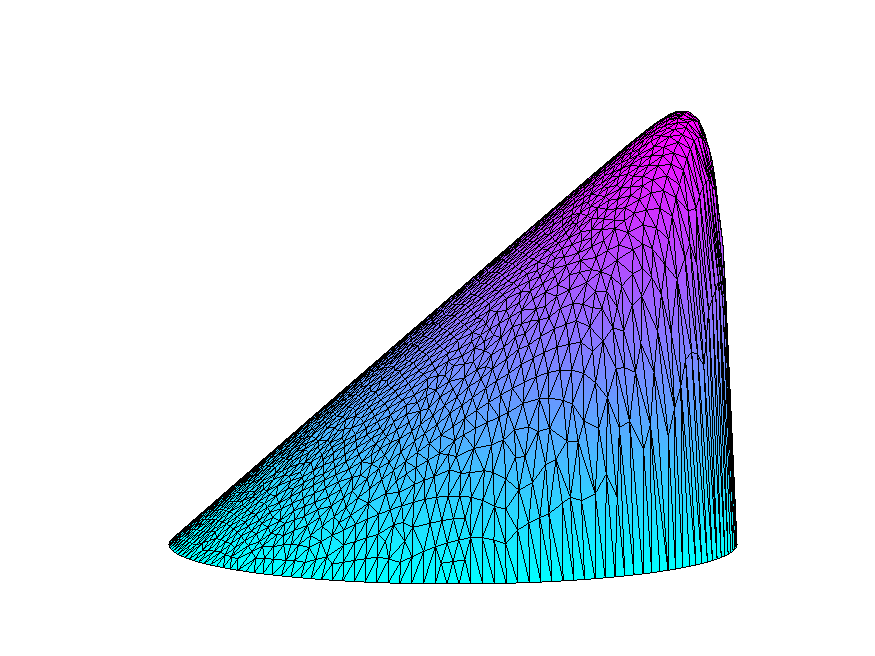}
    \caption{}
    \label{sol1}
    \end{subfigure}
    \hfill
    \begin{subfigure}[h]{.32\textwidth}
    \includegraphics[width=.98\textwidth,height=\textwidth]{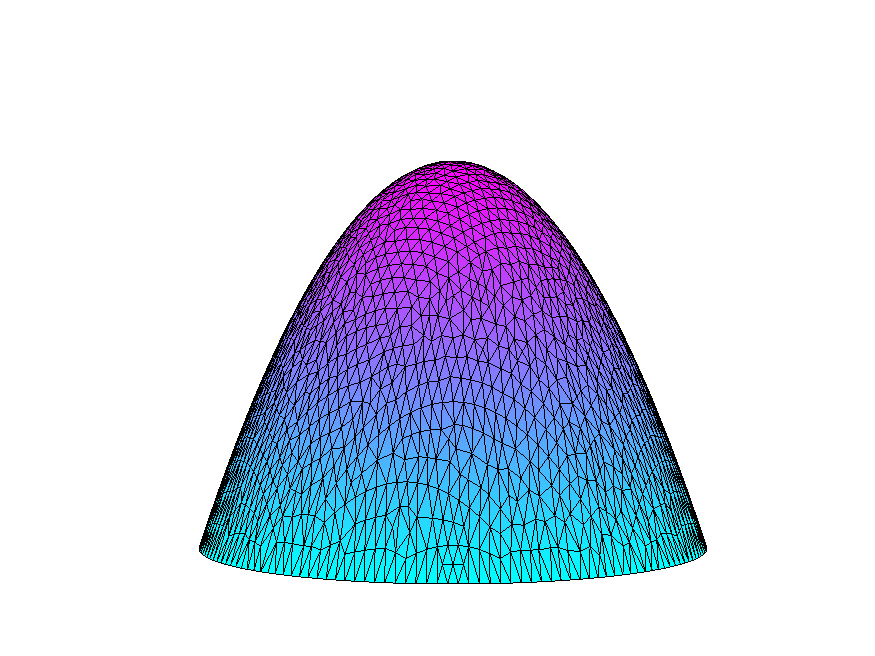}
    \label{sol2}
    \caption{}
    \end{subfigure}
    \hfill
    \begin{subfigure}[h]{.32\textwidth}
    \includegraphics[width=.98\textwidth,height=\textwidth]{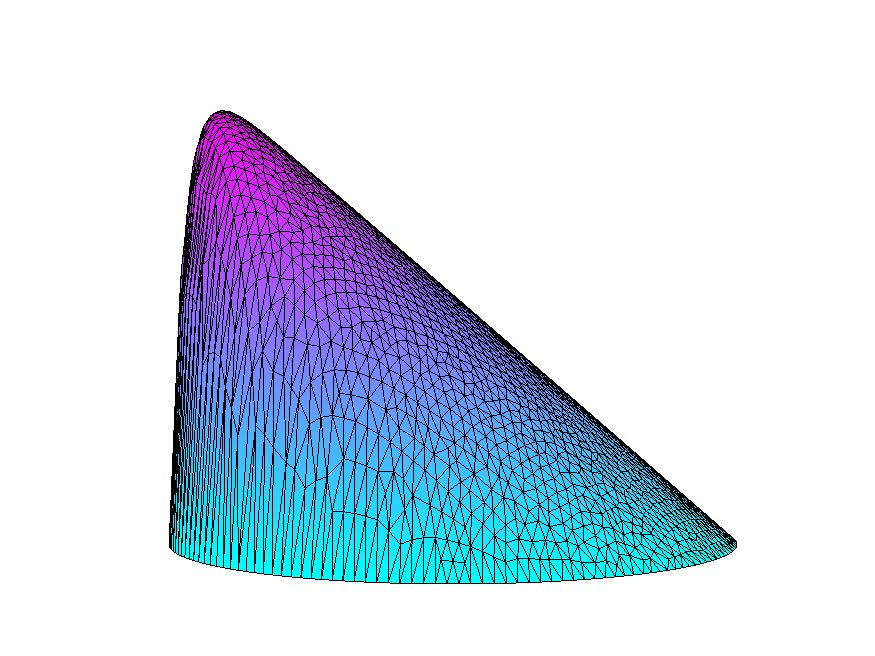}
    \label{sol3}
    \caption{}
    \end{subfigure}
    \caption{Solutions of Problem \eqref{FE prob} for different parameter values: (a) $\mu=(0,10)$, (b) $\mu=(0,0)$ and (c) $\mu=(\pi,10)$.}
    \label{sols}
\end{figure}
In Figure~\ref{sols}, three representative solutions are shown for different parameter values. It is evident that the solutions exhibit significantly different structures, which poses a challenge for linear reduced basis (RB) methods. 

We define three parameter domains
\begin{equation}
\mathcal{D}_{\mathrm{I}} := \{0\} \times [0, 10], \quad \mathcal{D}_{\mathrm{II}} := [0, \pi] \times \{10\}, \quad \mathcal{D}_{\mathrm{III}} := [0, \pi] \times [0, 10];
\end{equation}
we shall thus consider two cases:  $d = 1$ ($D_{\mathrm{I}}$  and  $D_{\mathrm{II}}$) or  $d = 2 $ ($D_{\mathrm{III}} $).

\begin{figure}[h]
    \centering
    \begin{subfigure}[h]{.48\textwidth}
    \includegraphics[width=.98\textwidth,height=.8\textwidth]{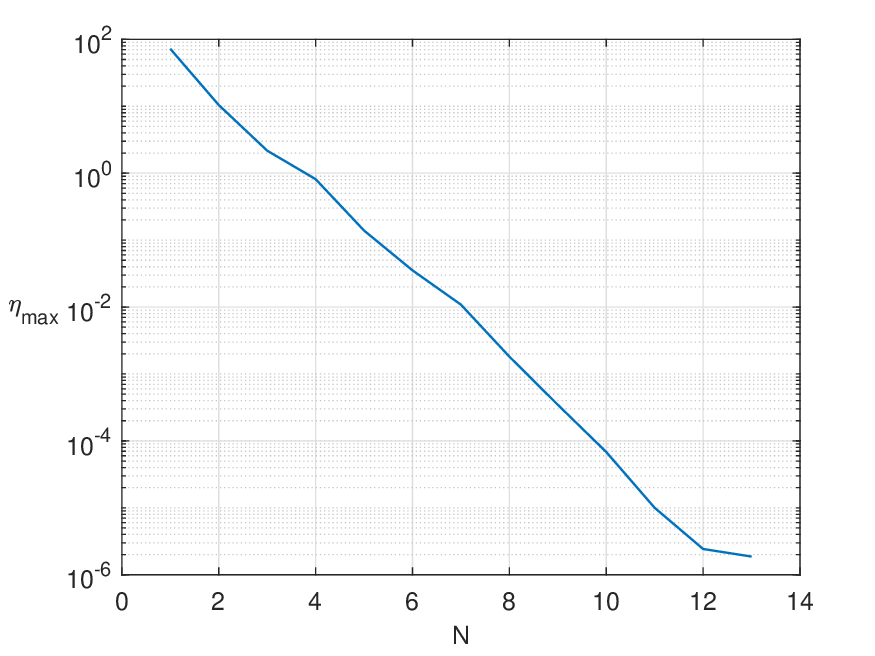}
    \caption{}
    \label{1d_i}
    \end{subfigure}
    %\hfill
    \begin{subfigure}[h]{.48\textwidth}
    \includegraphics[width=.98\textwidth,height=.8\textwidth]{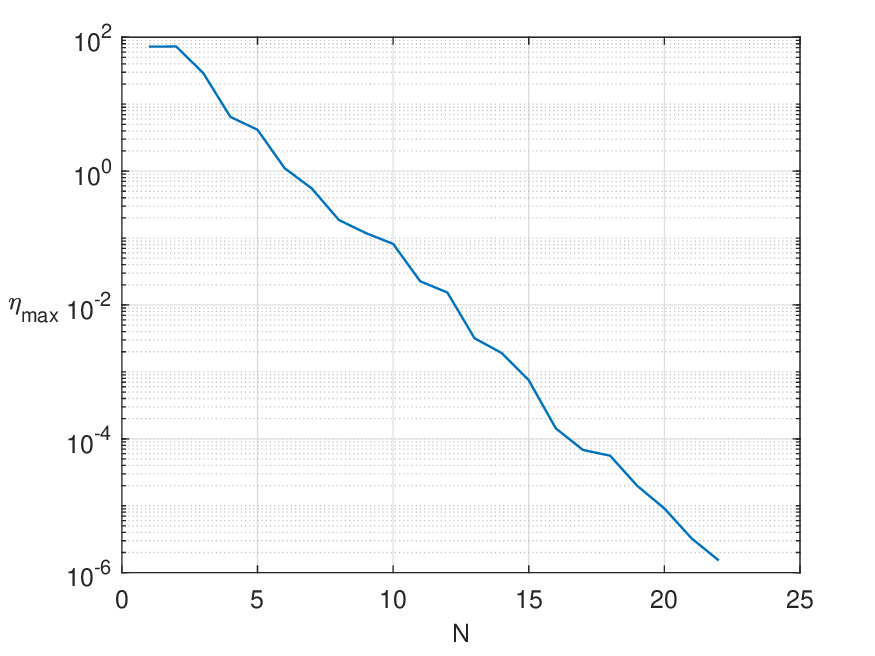}
    \label{1d_ii}
    \caption{}
    \end{subfigure}
    \caption{Linear RB convergence results: (a) one-parameter case $D_{\mathrm{I}}$. (b) one-parameter case $D_{\mathrm{II}}$.}
    \label{1d_linear}
\end{figure}

We introduce uniformly distributed random training sets $E_{\mathrm{I}} \subset \mathcal{D}_{\mathrm{I}}$, $E_{\mathrm{II}} \subset \mathcal{D}_{\mathrm{II}}$, and $E_{\mathrm{III}} \subset \mathcal{D}_{\mathrm{III}}$, of sizes $10^3$, $10^3$, and $10^4$, respectively. In Figure~\ref{1d_linear}, we plot $\eta_{\max}$ as a function of $N$ for the two one-parameter cases, $\mathcal{D}_{\mathrm{I}}$ and $\mathcal{D}_{\mathrm{II}}$. We observe that a very high accuracy can be achieved with a relatively small value of $N$ especially for the case $\mathcal{D}_{\mathrm{I}}$.

The parameters selected by the greedy algorithm are displayed in Figure~\ref{greed_selec_2}. Unlike in the previous problem, the resulting distribution is approximately uniform, reflecting the varying structure of the solution across the entire parameter domain. In Figure~\ref{linear_model_2}, we plot $\eta_{\max}$ versus $N$ for the two-parameter case $\mathcal{D}_{\mathrm{III}}$. The linear RB method fails to effectively approximate the solution with a small basis size. This is due to the significantly different solution structures corresponding to different parameter values, as illustrated in Figure~\ref{sols}. Consequently, the nonlinear approach is needed to effectively handle this problem.

\begin{figure}[h]
    \centering
    \begin{subfigure}[h]{.48\textwidth}
    \includegraphics[width=.98\textwidth,height=.8\textwidth]{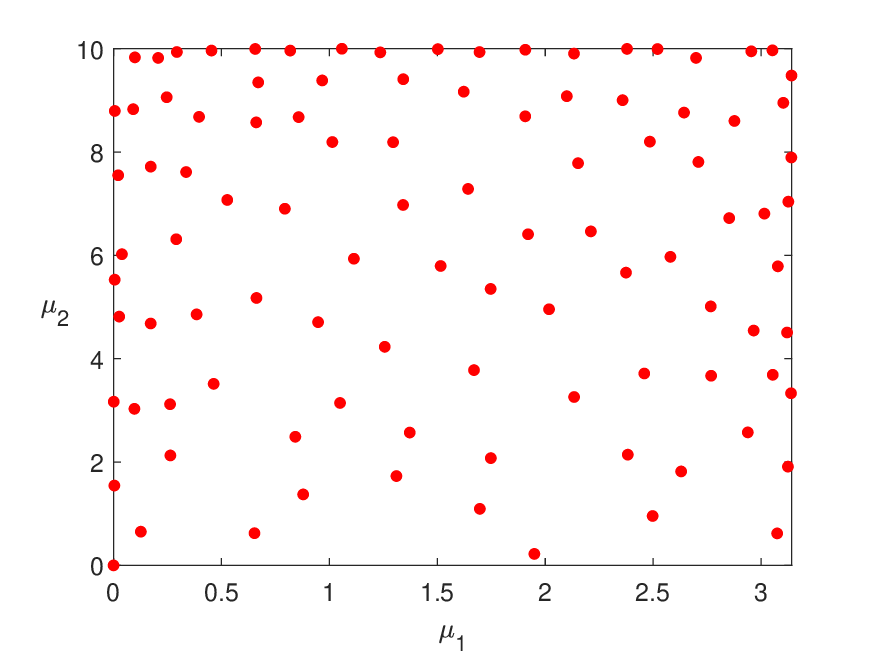}
    \caption{}
    \label{greed_selec_2}
    \end{subfigure}
    %\hfill
    \begin{subfigure}[h]{.48\textwidth}
    \includegraphics[width=.98\textwidth,height=.8\textwidth]{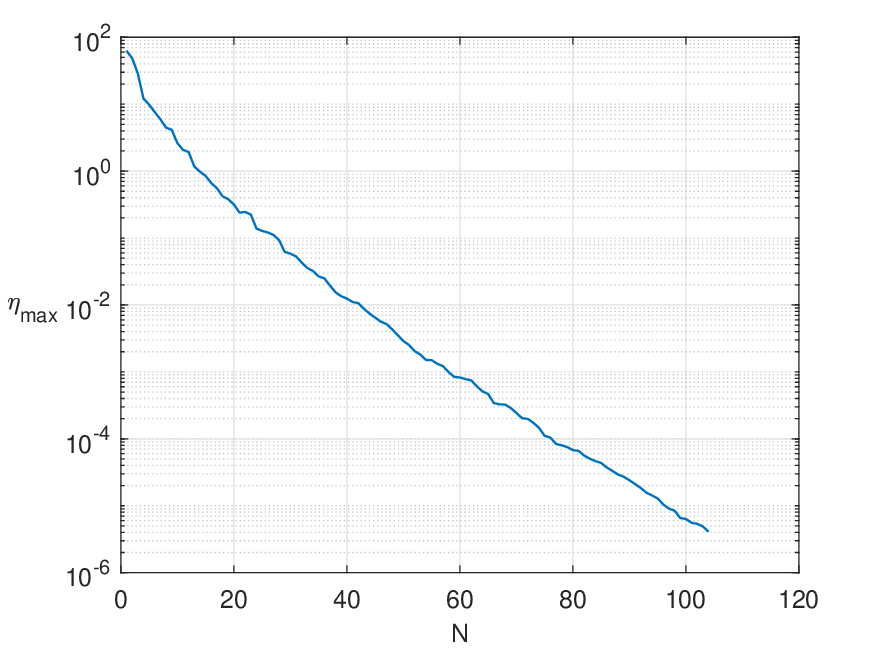}
    \caption{}
    \label{linear_model_2}
    \end{subfigure}
    \caption{(a) Greedy parameter selection for $D_{\mathrm{III}}$. (b) Linear RB convergence results for $D_{\mathrm{III}}$.}
\end{figure} 

Next, we present the convergence results of our algorithm in comparison with those of the algorithm described in \cite{eftang2010hp}.  We begin with the two one-parameter cases, $\mathcal{D}=\mathcal{D}_{\mathrm{I}}$ and $\mathcal{D}=\mathcal{D}_{\mathrm{II}}$. We choose $(0,0)$ as the initial parameter for the case $\mathcal{D} = \mathcal{D}_{\mathrm{I}}$ and $(0,10)$ for the case $\mathcal{D} = \mathcal{D}_{\mathrm{II}}$. The training set consists of $10^2$ randomly sampled points. In Figure~\ref{comp_1d}, we plot the number of subdomains $K$ against the tolerance $\epsilon$ for the two algorithms, considering three different cases: $N = 1, 2,$ and $3$. The observed convergence rates align well with the theoretical results presented in Theorem~\ref{th conv}. We observe that both algorithms yield comparable results, with our proposed algorithm demonstrating slightly superior performance in the case $N = 1$.

\begin{figure}[h]
    \centering
    \begin{subfigure}[h]{.49\textwidth}
    \includegraphics[width=\textwidth,height=.8\textwidth ]{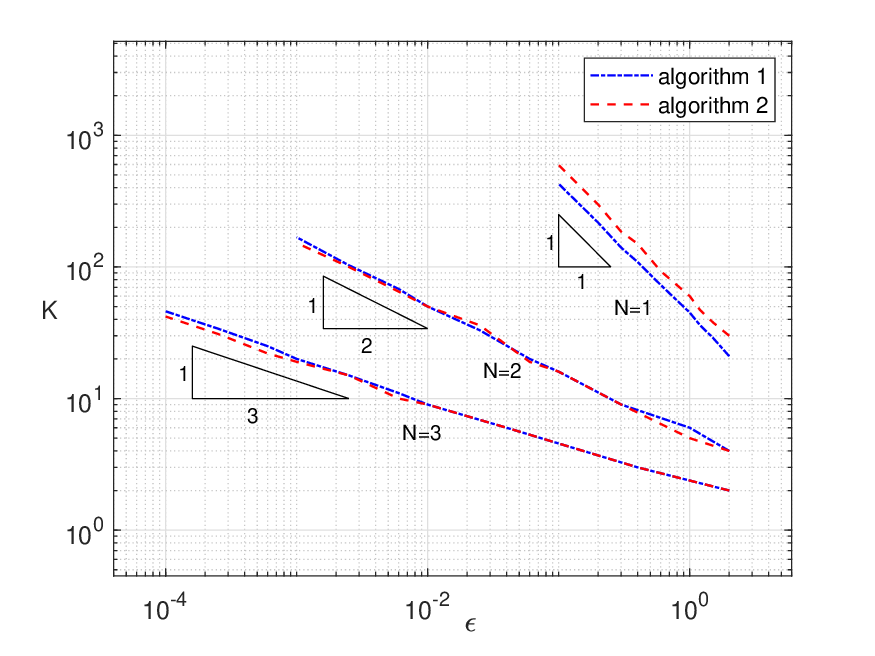}
    \caption{}
    \end{subfigure}
    %\hfill
    \begin{subfigure}[h]{.49\textwidth}
    \includegraphics[width=\textwidth,height=.8\textwidth ]{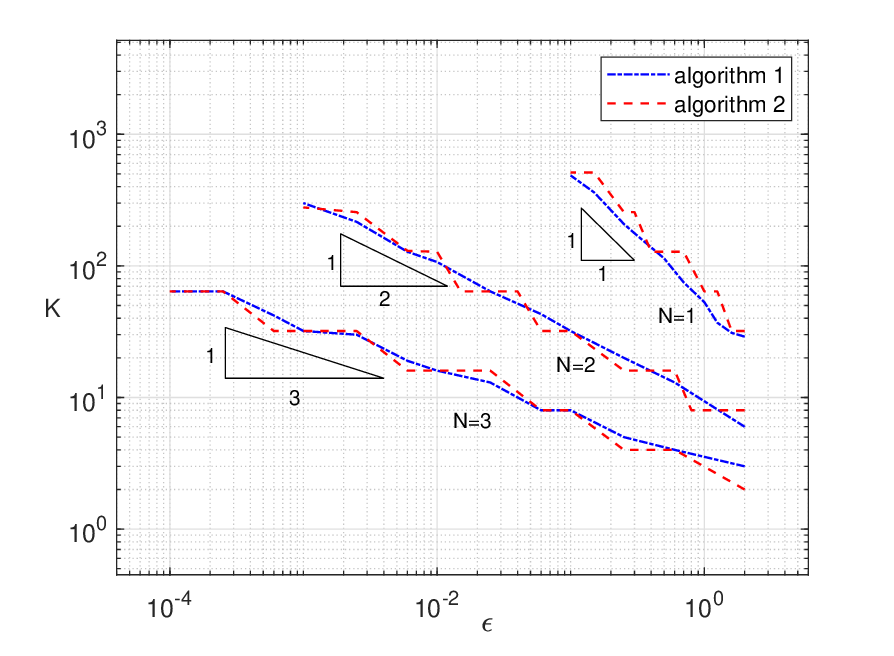}
    \caption{}
    \end{subfigure}
    \caption{Convergence results for the two algorithms at $N=1,2,$ and $3$ for the one parameter cases (a) $\mathcal{D}=\mathcal{D}_{\mathrm{I}}$, (b) $\mathcal{D}=\mathcal{D}_{\mathrm{II}}$.}
    \label{comp_1d}
\end{figure}

For the two-parameter case $\mathcal{D} = \mathcal{D}_{\mathrm{III}}$, the initial parameter is $(0,0)$, and the training set consists of $10^3$ randomly sampled points. Convergence results for different RB space dimensions, $N = 1, 4,$ and $16$, are presented in Figure~\ref{comp_2d}. Once again, the observed convergence rates are consistent with Theorem~\ref{th conv}. The results indicate that our proposed algorithm (Algorithm $1$) achieves better accuracy in certain cases, specifically when $N = 1$ and $N = 16$. On the other hand, Algorithm $2$ shows slightly better performance in the case $N = 4$.
\begin{figure}[h]
    \centering
    \includegraphics[width=.75\textwidth,height=.5\textwidth]{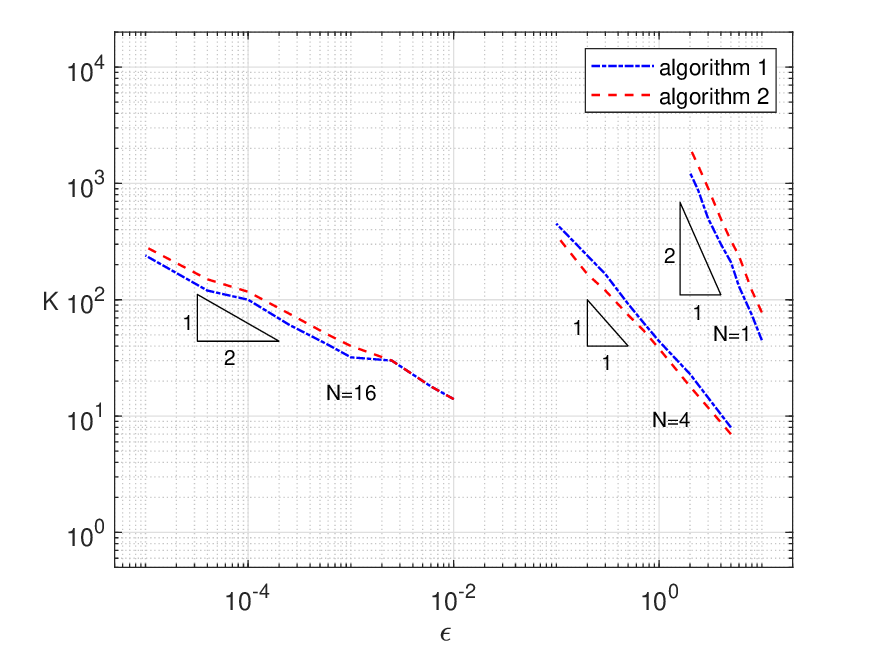}
    \caption{Convergence results for the two algorithms at $N=1,4,$ and $16$ for the two parameter case $D_{\mathrm{III}}$}
    \label{comp_2d}
\end{figure}

Figure~\ref{comp_partit} illustrates the partition of the parameter domain $\mathcal{D}$ for the two algorithms, here in the case $\mathcal{D} = D_{\mathrm{III}}$ with $N = 4$ and $\epsilon = 0.3$. Once again, the proposed algorithm yields a particularly simple partitioning structure, which translates into reduced online storage requirements.

\begin{figure}[h]
    \centering
    \begin{subfigure}[h]{.48\textwidth}
    \includegraphics[width=\textwidth,height=.8\textwidth ]{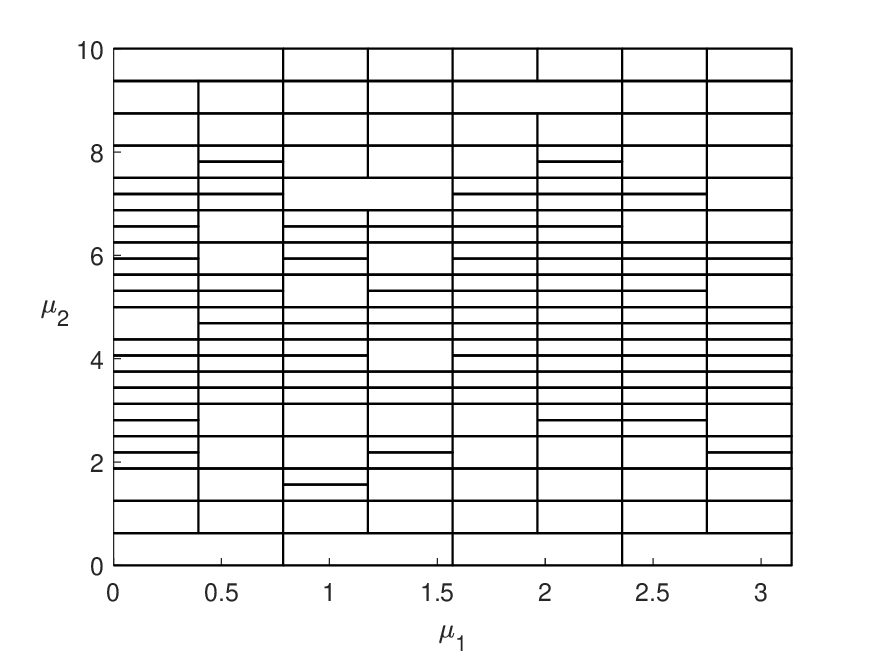}
    \caption{}
    \end{subfigure}
    %\hfill
    \begin{subfigure}[h]{.48\textwidth}
    \includegraphics[width=\textwidth,height=.8\textwidth ]{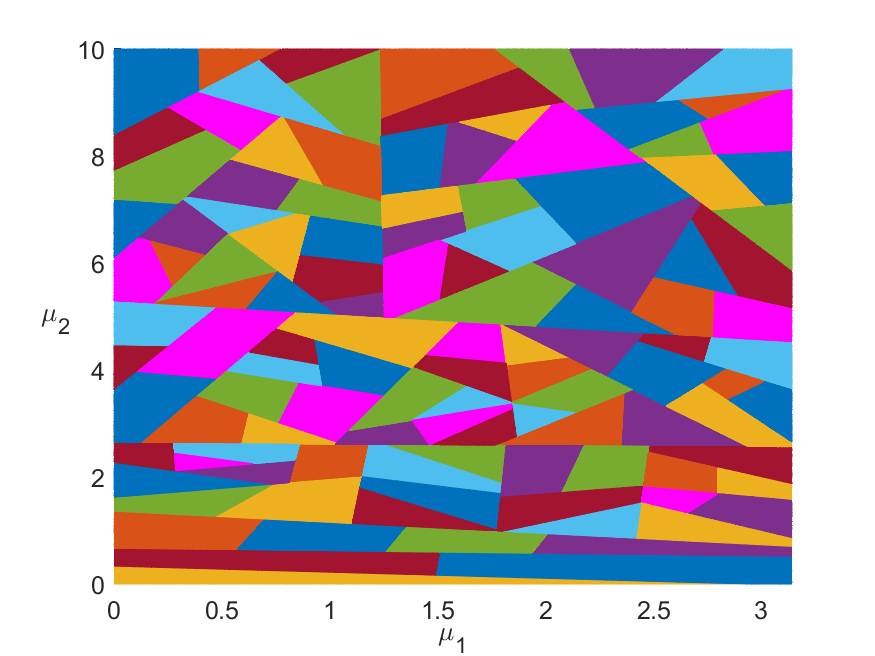}
    \caption{}
    \end{subfigure}
    \caption{Final parameter domain partition for the case $D_{\mathrm{III}}$ with $N=4$ and $\epsilon=0.3$ for (a) the proposed algorithm (b) the algorithm from \cite{eftang2010hp}.}
    \label{comp_partit}
\end{figure}

Compared to the algorithm presented in \cite{eftang2010hp}, the proposed algorithm offers several advantages. First, the parameter domain partition resulting from our method consists of subdomains with a tensor product structure. This structure enables the use of explicit volume formulas for each subdomain, which in turn allows for the development of rigorous convergence results. Moreover, the tensor product structure significantly reduces the online storage requirements, since the storage depends only on the bounding points of each subdomain. In contrast, the partition in \cite{eftang2010hp}, which is based on a proximity function, offers no control over the shape of the subdomains, leading to substantially higher storage needs. Furthermore, the proposed algorithm can be applied to a pre-partitioned parameter domain where the subdomains already exhibit a tensor product structure. This feature drastically reduces the computational cost of the offline stage required to achieve a given tolerance. Finally, with regard to accuracy, our algorithm achieves better results in terms of the number of subdomains needed to reach a prescribed tolerance in several instances, as demonstrated by the numerical results.

\section{Conclusion}

In this work, we have developed and analyzed a nonlinear RB method based on a binary-tree partitioning of the parameter domain into tensor-product-structured subdomains. Each subdomain is associated with its own local linear RB space, with dimension bounded by a prescribed value. The proposed splitting criterion, which always acts along the longest parameter direction, allows us to rigorously control subdomain geometry and derive explicit convergence bounds.

Our theoretical results establish convergence rates for the general case of arbitrary parameter domain dimension $d$ and RB space size $N$. In particular, we proved that, under some regularity and quasi-uniformity assumptions, the number of subdomains required to achieve a given tolerance $\epsilon$ scales as
\[
K(\epsilon) = \mathcal{O}\left( \epsilon^{-d/N^{1/d}} \right).
\]

Comprehensive numerical experiments on diffusion and convection-diffusion model problems confirm the theoretical predictions. The proposed method consistently exhibits the expected convergence rates and, in several cases, outperforms the nonlinear RB approach introduced in \cite{eftang2010hp} in terms of the number of subdomains needed for a prescribed accuracy. Furthermore, the tensor-product partition structure not only has a simple description which has low storage requirements but also allows for reducing computational cost by enabling the use of pre-partitioned parameter domains.

Future research directions include extending the framework to time-dependent and nonlinear PDEs, as well as refining the theoretical analysis to account for the adaptive nature of the method without restricting the partition to the quasi-uniform case.

\bibliographystyle{acm}
\bibliography{ref}

\end{document}